\documentclass[12pt]{article}

\usepackage{enumerate}
\usepackage{epsf,epsfig,amsfonts}

\usepackage{mathdots}
\usepackage{amsmath,amssymb,amsthm}
\usepackage{url}
\usepackage{mathtools}
\usepackage{epsfig,amsfonts,graphicx,color}
\newtheorem{Th}{Theorem}
\newtheorem{Lemma}{Lemma}[section]
\newtheorem{Rem}{Remark}[section]
\newtheorem{proposition}{Proposition}[section]
\newtheorem{Cor}{Corollary}[section]

\newtheorem{Ex}{Example}[section]

\newcommand{\be}{\begin{equation}}
\newcommand{\ee}{\end{equation}}

\newcommand{\tr}{\mathrm{tr}\,}

\newcommand{\R}{\mathbb{R}}
\newcommand{\Id}{\operatorname{Id}}

\newcommand{\pd}[2]{\frac{\partial#1}{\partial#2}}

\newcommand{\const}{\operatorname{const}}

\newcommand{\weg}[1]{}
\title{Applications of Nijenhuis geometry II: maximal  pencils of multihamiltonian structures    of hydrodynamic type}
\author{Alexey V. Bolsinov\footnote{ School of Mathematics,
 Loughborough University,
 LE11 3TU, UK;   Faculty of Mechanics and Mathematics, Moscow State University and Moscow Center
for Fundamental and Applied Mathematics,  119992,  Moscow  Russia \ \ 
 \quad {\tt A.Bolsinov@lboro.ac.uk} } \quad
\& \quad  Andrey Yu. Konyaev\footnote{Faculty of Mechanics and Mathematics, Moscow State University, 119992, Moscow Russia
 \ \ \quad {\tt  maodzund@yandex.ru}} \quad \& \quad Vladimir S. Matveev\footnote{
Institut f\"ur Mathematik, Friedrich Schiller Universit\"at Jena,
07737 Jena Germany  \ \ \quad {\tt  vladimir.matveev@uni-jena.de}} 
}  
\date{}
\begin{document}
\maketitle
%\baselineskip 24pt
%%%%%%%%%%%%%%%%%%%%%%%%%%%%%%%%%%%%%%%%%%%%%%%%%%%%%%%%%%%%%%%%%%

\begin{abstract} 

We connect  two a priori unrelated topics, theory of geodesically equivalent metrics in differential geometry, and theory of compatible infinite dimensional Poisson  brackets of hydrodynamic type in mathematical physics.   Namely, we prove  that a pair of geodesically equivalent metrics such that one is flat produces a pair of such brackets. We construct Casimirs for  these brackets and the corresponding commuting flows.  There are two ways to produce a large family of compatible Poisson structures from a pair of geodesically equivalent metrics one of which is flat. One of these families  is $(n+1)(n+2)/2$ dimensional; we describe it completely  and  show that it is maximal. Another  has dimension $\le n+2$ and is, in a certain sense, polynomial.
We show that a nontrivial  polynomial family of  compatible  Poisson structures of dimension $n+2$ is unique and comes from a pair of geodesically equivalent metrics.   
In addition, we generalise a result of Sinjukov (1961)  from constant curvature metrics to  arbitrary Einstein metrics.  
\end{abstract}

{\bf MSC classes:}  37K05, 37K06, 37K10,  37K25,   37K50,  53B10,  53A20,  53B20,  53B30, 53B50, 53B99,  53D17

 \tableofcontents

\section{Introduction}

This paper continues the {\it Nijenhuis Geometry} programme started in \cite{Nijenhuis1} and further developed in \cite{ Nijenhuis3, NijenhuisAppl1,Nijenhuis2}.  This programme was initially motivated by the fact that Niejnhuis operators (i.e. fields of endomorphisms $L=(L^i_j)$ with vanishing Nijenhuis torsion \cite{nij}) naturally appear in a number of different areas of geometry, algebra and mathematical physics.  For this reason their normal forms, singularities and global properties deserve more systematic study than before.  Although Nijenhuis operators usually serve as an auxiliary object  in various mathematical constructions (two of them will be discussed in the present paper), in many cases their role is crucial and we intend to demonstrate this: the reader will  notice that in many computations below we use $L$ as a primary object and this leads to essential simplifications and new results.

We  believe that the appearance of Nijenhuis operators in various, seemingly unrelated research areas, might be an evidence of a hidden relationship between them. Indeed, one may often observe similarity of ideas, techniques or clever tricks used therein.  Sometimes the relationship is much deeper and is manifested  in ``overlapping''  at the level of mathematical objects studied in these areas. In the present paper, we demonstrate such an overlap between geodesically equivalent pseudo-Riemannian metrics and compatible Poisson brackets of hydrodynamic type (see Theorem \ref{thm:main1}, \ref{thm:main2} and \ref{thm:main3} below).  Once such a relationship is established and understood, one may try to transfer insights from one area to the other.  That is what we do in Theorem \ref{thm:main4} by using our expertise in the theory of geodesically equivalent metrics to prove a uniqueness result for a certain type of Poisson pencils of hydrodynamic type.  

Also we would like to emphasise that general methods of Nijenhuis Geometry allows one to deal with singularities, i.e.,  those points where collisions of eigenvalues happen and $L$ changes its algebraic type. Traditionally, such singularities are excluded and as a consequence  most  constructions are local and restricted to a domain where $L$ reduces to some standard canonical form (for instance, is $\R$-diagonalisable with simple spectrum). In many problems, however, singularities cannot be ignored. First of all,  this  relates to global problems, when $L$ ``lives'' on a closed manifold and singularities  become unavoidable.  In this paper, we either impose no additional restrictions on $L$, or assume that $L$ is differentially non-degenerate (this condition still allows singularities).  Moreover,  we give a description of all Nijenhuis operators geodesically compatible with a flat metric. Since this description is explicit, it can be used for analysis of  singularities  these operators may have.

{\bf Acknowledgements.}  We thank Jenya Ferapontov  for his valuable comments and explanations.  
The most essential steps resulted in this paper would not have been done without outstanding research environment offered to us by Centro Internazionale per la Ricerca Matematica, Trento and the Institute of Advanced Studies, Loughborough University  where the  first results of the paper were obtained, and also Centre International de Rencontres Math\'ematiques Lumini, where  the  paper was finished and 
written. 
We are also grateful to Jena Universit\"at, in particular, Ostpartnerschaft programm for supporting our research on Nijenhuis Geometry for several years. The work of Alexey Bolsinov and Andrey Konyaev was  supported by Russian Science Foundation (project 17-11-01303).

\section{Basic definitions and main results}\label{sect2}
We start with introducing two classes of objects we will be dealing with:  Poisson structures of hydrodynamic type and geodesically equivalent metrics. 

Given  a   metric $g$ of any signature  on a manifold $M$  which is always assumed to  be 
of  dimension $n\ge 2$, for 
 a function $h$, treated as the density of a Hamiltonian of hydrodynamic type,  one can construct an operator (=\,$(1,1)$-tensor field)  by the formula
\begin{equation} \label{eq:ham}
h\mapsto \nabla^i\nabla_j h= g^{is}\tfrac{\partial^2 h}{\partial x^s \partial x^j} - \Gamma^{is}_{j} \tfrac{\partial h}{\partial x_s}  
\end{equation} (we use $g$ for index manipulations, for example $\Gamma^{is}_{j} = \Gamma^s_{pj} g^{pi}$). 

If the metric is flat,  this construction has many parallels with the construction of the  Hamiltonian vector field by  a function $H$ and a Poisson bracket. Actually, a flat metric defines an infinite-dimensional Poisson bracket,  the density $h$ defines a  variational functional of hydrodynamic type and  the operator $\nabla^i\nabla_j h$ determines the Hamiltonian flow generated by this functional. We refer to \cite{ Doyle,Dubrovin-Novikov,Gelfand-Dorfman, Mokhov_symplectic}  for details.

Parallels with  and intuition coming from finite-dimensional Poisson brackets  appeared to be  very helpful in studying the following systems of  partial differential equations: we view local 
coordinates $x^{1},\dots,x^{n}$ on $M^n$ as unknown functions which depend on two variables $t$ and $\tau$ and consider the following quasilinear systems of $n$ PDE on $n$ unknown functions of two coordinates:   
\begin{equation}\label{sys:hidro}
\frac{\partial }{\partial  t}x^i(t,\tau) = A^i_j(x) \frac{\partial  }{\partial \tau}x^j(t,\tau), \quad\mbox{where }  A^i_j = \nabla^i\nabla_j h. 
\end{equation}
Such systems are called {\it systems  of hydrodynamic type} in literature.

This construction was generalised in \cite{Ferapontov-Mokhov,Ferapontov-viniti} for metrics of constant curvature $K$: in this case   the analog of formula  \eqref{eq:ham} is 
\begin{equation} \label{eq:ham1}
h\mapsto \nabla^i\nabla_j h + K h \,\delta^i_j.    
\end{equation}
and correspondingly $A$ in \eqref{sys:hidro} is given by  $A^i_j(x) = \nabla^i\nabla_j h + K h \,\delta^i_j$. 

Recall that two Poisson structures are {\it compatible}, if their sum is also a Poisson structure. In our  infinite-dimensional 
situation,  compatibility of two Poisson structures  coming from   constant   curvature 
metrics $g$ and $\bar g$ 
is  equivalent to the following two conditions (see  e.g. the survey  \cite{Mokhov}):
 \begin{itemize}\item[(A)] 
   the operator  $L^{i}_{j}:= \bar g^{is} g_{sj}$ is a {\it Nijenhuis operator}, i.e., its Nijenhuis torsion vanishes, that is,
   $$
   \mathcal N_L(u,v) = L^2[u,v] - L[Lu, v] - L[u,Lv] + [Lu,Lv] = 0
   $$
   for arbitrary vector fields $u$ and $v$  \cite{nij}.

		\item[(B)] For any  $\alpha, \beta \in \mathbb{R}$  
		such that the operator  $\alpha  \Id + \beta L$  is {\it invertible}, 
		the metric  $\hat g= g(\alpha  \Id + \beta L)^{-1}$  has constant curvature  $\hat K = \alpha K + \beta \bar K$ where $K$ and $\bar K$  are the curvatures of $g$ and $\bar g$ respectively.
\end{itemize}

Notice that by \cite{Mokhov1,Mokhov}  condition (A)  is equivalent to the  following property: 
for any  $\alpha, \beta \in \mathbb{R}$ such that $\alpha \Id + \beta L$ is non-degenerate,  the Christoffel symbols of the  metric 
$\hat g= g(\alpha \Id + \beta L)^{-1}$  are given by $\hat \Gamma^{ij}_{k}=\alpha \,\Gamma^{ij}_{k} +   \beta \,\bar \Gamma^{ij}_k$. In the last formula, we raise the index $i$  of the Christoffel coefficients $\Gamma$ (resp.  $\bar \Gamma$,  $\hat \Gamma$) 
of the metric $g$  (resp. $\bar g,$ $\hat g$) by its  own metric.

We will call   a set of  metrics {\it Poisson-compatible}, if any two metrics $g$ and $\bar g$ from this set satisfy (A,B).

Similar to the finite-dimensional case, 
the existence of a nontrivial  compatible Poisson structure  provides additional tools to analyse the system \eqref{sys:hidro}, in particular to construct explicit solutions  and to study long-time behaviour of solutions. 
We refer to   \cite{Ferapontov-viniti,  Mokhov,Tsarev} for details.

Nijenhuis operators appeared naturally  in the theory of geodesically equivalent metrics. Recall that two metrics $g$ and $\tilde g$  on one manifold   are \emph{geodesically equivalent}, if they have the same geodesics considered as unparametized curves.  
Let us  consider the $(1,1)-$tensor $L=L(g,\tilde g)$ defined by
\begin{equation} \label{L}
L_j^i := \left(\frac{|\det(\tilde g)|}{|\det(g)|}\right)^{\frac{1}{n+1}} \tilde g^{ik}
g_{kj},\end{equation}
where $\tilde g^{ij}$ is the  (automatically, symmetric)  tensor dual to $\tilde g_{ij}$ (in the sense $\tilde g^{is}\tilde g_{js}=\delta^i_j$). 
It is known \cite{BoMa2003} that for geodesically equivalent metrics the operator $L$ is Nijenhuis.   Notice that $\tilde g$ can be recovered from $g$ and $L$ as
\begin{equation}\label{L2}
\tilde  g = \frac{1}{|\det L|} \, gL^{-1},
\end{equation}
which is equivalent to \eqref{L}.

For a given $g$, we will call an 
operator $L$ {\it geodesically compatible with $g$},  if   \eqref{L2} defines a metric $\tilde g$ geodesically equivalent to $g$.

 This condition is equivalent to the property   that $L$ is self-adjont with respect to $g$, is non-degenerate,  and   satisfies the system of PDEs \eqref{eq:L} below. Sometimes speaking about geodesic compatibility of $L$ with $g$ one allows  $L$   to be degenerate and only requires that $L$ is selfadjoint and satisfies \eqref{eq:L}. Actually since $\Id= \delta^i_j$ is geodesically 
compatible  with every $g$ and  equation \eqref{eq:L} is linear, addition of $\const\cdot   \Id  $  to $L$ makes it locally non-degenerate.  In order to avoid misunderstanding, we will always either 
explicitly  require  the operator $L$ to be non-degenerate or allow it to be degenerate.

Our first result is the following relation between geodesically equivalent  and Poisson-compatible  metrics:

\begin{Th} \label{thm:main1}  Let  $g$ be a flat metric.   Then the following statements hold:
 
\begin{enumerate} 
 \item If   $L$ is geodesically compatible with $g$ and non-degenerate, then the metric  $gL^{-1}$ has constant (possibly zero) curvature and  is Poisson-compatible with $g$.

\item  If non-degenerate  $L_1$ and $L_2$ are  geodesically compatible with $g$, 
then the metrics $gL_1^{-1}$ and $g L_2^{-1}$ are Poisson-compatible.

\end{enumerate}
  \end{Th}

The first statement of this theorem means that the metrics $g$ and $\bar g=gL^{-1}$  define a family (pencil) of compatible Poisson brackets of hydrodynamic type, which can be used to construct integrable systems of quasilinear PDEs.  One of such constructions is based on Casimir functions of the relevant Poisson structures.  In the context of this paper,  by a Casimir of the Poisson structure of hydrodynamic type (related to a metric $g$ with constant curvature $K$, see Remark \ref{rem:convention} for the definition of $K$)  we understand a function  $h:M\to \mathbb{R}$ which, if plugged into \eqref{eq:ham1}, produces the  zero operator:
$$
\nabla^i  \nabla_j h + K h \,\delta^i_j = 0.
$$
Since the property of $h$ to be a Casimir  is completely defined in terms of $g$, for the sake of brevity,  we will refer to it as {\it a Casimir of $g$}.   
 If $K=0$ and $x^1,\dots, x^n$ are flat coordinates for $g$, then the Casimirs are just linear combinations of the form  $a_0 + a_1x^1 + \dots + a_nx^n$,  $a_i\in\R$.    If $K\ne 0$,  the Casimirs still form a vector space of dimension $n+1$ that admits a simple explicit description as soon as $g$ is given by means of a certain {\it canonical} model. For instance, if $g$ is the standard metric on the round sphere $S^n \subset  \R^{n+1}$, then the Casimirs of $g$ are just restrictions of linear functions from $\R^{n+1}$ to $S^n$.  
 
Let  $g$ and $\bar g$ be two Poisson-compatible metrics and 
$h_\alpha$ be a Casimir of an arbitrary metric $g_\alpha=(g^{-1} - \alpha \bar g^{-1})^{-1}$, $\alpha \in \R$  from the corresponding pencil.   It is a well known fact that the Hamiltonian flows generated by  $h_\alpha$'s w.r.t. the Poisson bracket related to $g$ commute and this property leads to their integrability.  However, if $h_\alpha$ is a common Casimir of   $g_\alpha$ and $g$, then the corresponding flow vanishes and the above construction becomes trivial.  For this reason it makes sense to distinguish {\it essential} Casimirs $h_\alpha$.  The next theorem describes Casimirs and the corresponding commuting Hamiltonian flows for the pencil of Poisson-compatible metrics from Theorem~\ref{thm:main1}.

\begin{Th}\label{thm:main2}
Let  $g$ be a flat metric and $L$ be  a non-degenerate geodesically compatible operator. Then, the following statements hold:  

\begin{enumerate}

\item The metrics $g$ and $\bar g=gL^{-1}$ have $n$ common independent  Casimirs.

\item For any $\alpha\in \mathbb{R}$  such that $\det(-\alpha L + \Id)> 0$, the function $h_\alpha=\sqrt{\det(-\alpha L + \Id)} $ is  a Casimir of the metric $g(-\alpha L + \Id)^{-1}$.   

\item The Hamiltonian flow generated by $h_\alpha$ w.r.t. the Poisson bracket related to $g$ is given (up to a factor depending on $\alpha$) by the operator    
\begin{equation} 
\label{eq:flow} 
 A_\alpha = \frac{1}{\sqrt{\det(-\alpha L + \Id)}}\left( -\alpha L + \Id\right)^{-1}.\end{equation}
 The flows given by $A_\alpha$ and $A_\beta$ commute for all $\alpha,\beta\in \R$.
 \end{enumerate}
\end{Th}

\begin{Rem}\label{rem:0}{\rm
Instead of commuting flows given by \eqref{eq:flow} and parametrised by $\alpha\in\R$, one usually considers the coefficients of the expansion of  \eqref{eq:flow} in powers of $\alpha$:
$$
\frac{1}{\sqrt{\det(-\alpha L + \Id)}}\left( -\alpha L + \Id\right)^{-1} = \Id + \alpha A_1 + \alpha^2 A_2 + \alpha_3 A_3 + \dots,
$$
with
$$
\begin{aligned}
A_1 &= L + \frac{1}{2} \sigma_1 \Id, \\
A_2 &= L^2 + \frac{1}{2} \sigma_1 L + \left( \frac{1}{2} \sigma_2 + \frac{3}{8}\sigma_1\right)\Id,   \\
A_3 &= L^3 + \frac{1}{2} \sigma_1 L^2 + \left( \frac{1}{2} \sigma_2 + \frac{3}{8}\sigma_1\right) L + 
\left(\frac{1}{2}\sigma_3 + \frac{3}{4} \sigma_1\sigma_2 + \frac{5}{16}\sigma_1^3\right)\Id, \\
& \dots 
\end{aligned}
$$  
where $\sigma_k$'s denote the coefficients of the characteristic polynomial $\det (\Id - \alpha L) = 1 -\sigma_1 \alpha -\dots -\sigma_n\alpha^n$.  The flows generated by these operators commute and hence define an integrable system of PDEs.  They admit an infinite series of conservation laws given by the Casimir functions $h_\alpha = \sqrt{\det(-\alpha L + \Id)}$  or, equivalently, by the coefficients of the $\alpha$-power expansion 
$h_\alpha = 1 + \alpha F_1 + \alpha^2 F_2 + \dots$:
$$
F_1 = -\frac{1}{2} \sigma_1, \quad F_2 = -\frac{1}{2} \sigma_2 -\frac{1}{8} \sigma_1^2 , \quad F_3 = -\frac{1}{2}\sigma_3 -\frac{1}{4}\sigma_1\sigma_2 - \frac{1}{16}\sigma_1^3, \quad \dots
$$

Notice that the commuting flows and their conservation laws are defined in terms of a Nijenhuis operator $L$ only, whereas the metric $g$ in not involved in the final conclusion.  In fact, this integrable system is a particular example of the so-called integrable $\varepsilon$-systems introduced and studied by M.\,Pavlov in \cite{Pavlov}  for $\R$-diagonalisable operators $L$.  These systems, in turn, are a particular case of a nice construction developed by P.\,Lorenzoni and F.\,Magri in \cite{mag2} and based exclusively on a Nijenhuis operator $L$.  No other ingredient is needed so that 
the system itself and its integrability can naturally be understood in the framework of Nijenhuis Geometry. 
}\end{Rem}

%It is known that for flat metrics the dimension of this vector space admits its maximal  value   $\tfrac{(n+1)(n+2)}{2}$ (locally or on  simply-connected manifolds). We re-prove this result and give an explicit  description of geodesically compatible pairs $(g, L)$ with flat $g$ in \S \ref{appendix}.

If we allow $L$ to be degenerate,  then the operators geodesically compatible with $g$ form a vector space $\mathcal L_g$.  Since all these operators $L$ are Nijenhuis, we may refer to $\mathcal L_g$ as a {\it Nijenhuis} pencil.  According to the second statement of  Theorem  \ref{thm:main1} the pencil $\mathcal L_g$,  in the case of a flat metric  $g$, automatically leads to a {\it large} family of Poisson-compatible constant curvature metrics of the form $gL^{-1}$  (it is  more appropriate here  to consider contravariant metrics; they are given by 
the matrices $Lg^{-1}= L^{ij}$). The following theorem provides an explicit description for them.
 
 \begin{Th}\label{thm:main3}
  Let $g$ be a flat metric and $x^1,\dots,x^n$ be local coordinates in which all the components of $g$ are constant (i.e., local flat coordinates). 

Then every operator $L$ geodesically compatible to $g$ with the index raised by $g$  is given by the following formula:  
    \begin{equation}\label{norm}
        L^{ij} = a^{ij} + b^i x^j + b^j x^i - K x^i x^j.
    \end{equation}
		Here $(a^{ij}, b^i, K)$ are constants  and  $a^{ij}=a^{ji}$.  Conversely, every $L$ given by \eqref{norm}  is geodesically compatible 
	to  $g$ and the components of $\tfrac{1}{2} \operatorname{grad}_g (\tr L)$ are $\lambda^i=b^i- Kx^i$.

Moreover, near those points where $L^{ij}$ is non-degenerate and, therefore, defines a pseudo-Riemannian contravariant  metric, the curvature of this metric is constant and equals $K$. 		
    \end{Th} 

Following  \cite{Mokhov,Mokhov1}, we call a set of metrics {\it compatible},  if for any two of them, say $g$ and $\bar g$,  the operator  $L^{i}_{j}:= \bar g^{is} g_{sj}$ is {\it Nijenhuis} (i.e., condition (A) from the definition of Poisson compatibility is fulfilled) and in addition  for any $\alpha, \beta \in \mathbb{R}$  such that the operator  $\alpha  \Id + \beta L$  is non-degenerate, the curvature tensors  of $g$, $\bar g$  and $\hat g= g(\alpha  \Id + \beta L)^{-1}$  satisfy the relation 
\begin{equation} 
\label{eq:additivitycurvature} 
{ {\hat R^{ij}}}_{\  \ k\ell}  = \alpha  {R^{ij}}_{k\ell}  + \beta {{\bar R^{ij}}}_{\  \ k\ell}.
\end{equation} 
		(In each curvature tensor we raised the index by its own metric).
		
Comparing this definition with the definition of Poisson-compatibility  we see that Poisson compatible metrics are precisely compatible metrics of constant curvature.  

\begin{Cor}  \label{cor:compatible} 
Let $g=(g_{ij})$ be a flat metric on $\mathbb{R}^n$  and $L$ an operator geodesically compatible with $g$. 
Then for any polynomials $P$ and $Q$, the metrics $gP(L)^{-1}$ and $gQ(L)^{-1}$ are compatible (whenever the operators $P(L)$ and $Q(L)$ are invertible). In particular, if the metrics $g,gL^{-1}, \dots, gL^{-k}$ are of  constant curvature (resp. flat),  then 
$\{ P(L)g^{-1},  \ \deg P\le k\}$ is a pencil of Poisson-compatible (resp. flat) contravariant metrics.
\end{Cor}

Notice that every set  $\{ g_0, g_1, g_2, g_3, \dots \}$ of Poisson-compatible metrics (or equivalently the corresponding pencil of compatible Poisson structures of hydrodynamic type) leads to a natural Nijenhuis pencil which is the linear span of all Nijenhuis operators of the form  $L_k = g^{-1}_k g_0$, $k=0, 1,\dots$.   In a very similar way,  Nijenhuis pencils appear in the area of geodesically  equivalent metrics  and
 finite-dimensional multi-Hamiltonian structures and already for this reason they deserve to be studied as a separate subject in geometry.  For instance, it would be interesting to describe {\it maximal} Nijenhuis pencils, i.e., those which are not contained in any larger Nijenhuis pencil.   Notice that maximality of a Nijenhuis pencil would immediately imply maximality of the corresponding Poisson pencils and families of geodesically equivalent metrics (the converse, as a rule,  is not true).

In this context, the pencil from Theorem \ref{thm:main3} admits the following interpretation.  Consider an Euclidean metric $g\simeq \sum\mathrm{d} (x^i)^2$ and try to construct a big family of Poisson-compatible metrics containing $g$.  Let us start with the trivial family of metrics whose components are constant in coordinates $x^1,\dots, x^n$.  The corresponding Nijenhuis pencil consists of symmetric operators with constant entries.  Can this pencil be extended?  And if yes, then how?  The answer is that such an extension exists, is unique in the class of $g$-symmetric operators and coincides with the pencil from Theorem \ref{thm:main3}  (Corollary \ref{cor:bols_1}).
In particular, the family of Poisson compatible metrics from Theorem \ref{thm:main3}     is maximal (Corollary \ref{cor:bols_2}). 

Theorem \ref{thm:main3} combined with Theorem \ref{thm:main1}  gives us many examples of Poisson-compatible metrics and, consequently, integrable systems of hydrodynamic type. We discuss  these examples in Section \ref{sec:proof of theorem einstein}. Some of these examples has features which were not observed in integrable systems of hydrodynamic type before. In particular, we construct nontrivial  examples with Jordan blocks corresponding to nonconstant eigenvalues.

Corollary \ref{cor:compatible}  suggests one more method for constructing multidimensional pencils of Poisson brackets of hydrodynamic type. Let us start with a flat metric $g$ and an operator $L$ geodesically compatible to it.  We know that $g$ and $gL^{-1}$ are Poisson-compatible (and hence generate a pencil of dimension 2).    By \cite[Theorem 1]{Sinjukov1} (we reprove it in  Corollary \ref{cor:1}) $gL^{-1}$ and $L$ are still geodesically compatible. Hence if $gL^{-1}$ is flat, we can repeat this procedure and consider the metric $gL^{-2}$ which will be Poisson-compatible with $g$ and  $gL^{-1}$ by Corollary \ref{cor:compatible}.  If this metric is still flat, then we can make one more step and so on.  If $gL^{-k}$ is flat but  $gL^{-k-1}$ has non-zero constant curvature, then we show that  the process stops as the next metric $gL^{-k-2}$ does not  have constant curvature:

\begin{Th} \label{thm:einstein}
Let $L$ be geodesically compatible to an Einstein metric $g$ of constant non-zero  scalar curvature and non-degenerate. 
If $\mathrm{d}\operatorname{tr}(L)\ne 0$, then the metric $gL^{-1}$ is not an Einstein metric of constant scalar curvature. 
\end{Th} 

In Theorem \ref{thm:einstein} we allow any dimension $n\ge 2$.  In dimensions  $2$ and $3$,  Einstein metrics with constant scalar curvature are just metrics of constant curvature. In dimensions $\ge 3$ the scalar curvature of any Einstein metric is automatically constant. 

In dimension $n\ge 3$, the analog of Theorem \ref{thm:einstein} for metrics of constant curvature (which is in fact enough for our study of  metrics  $gL^{-k}$) is due to N. Sinjukov \cite{Sinjukov1}.
In dimension  $n\ge 3$, our proof goes along the same lines as that by Sinjukov (which was merely sketched), but essentially uses results which were    not available in the time of \cite{Sinjukov1}.  The proof in dimension $n=2$ is different, is partially based on the technology we develop in the present paper  and as far as we know is  new.

The example of a sequence    ($g, gL^{-1},\dots, gL^{-k}$ flat, $gL^{-k-1}$ of constant non-zero curvature) with  $k=n$ was constructed in  \cite{Ferapontov-Pavlov} and is as follows: 
the metric $g$ and  operator $L$ are  given in local coordinates $x^1,\dots, x^n$ by
\begin{equation}\label{ex:LC} 
g= \sum_{i=1}^n \prod_{j\ne i}  (x^i-x^j) (\mathrm{d}x^i)^2 \quad \textrm{and} \quad  L= \operatorname{diag}(x^1,\dots,x^n).
\end{equation}

One can slightly generalise this example to  include the  points at which the operator $L$ is not diagonalisable. The generalisation is essentially due  to    \cite{Nijenhuis1}:   the metric is \cite[Eq. (37)]{Nijenhuis1} and  the operator  is  \cite[Eq. (36)]{Nijenhuis1}. We repeat these  formulas below: 
\begin{equation} \label{eq:gL}
g^{-1}= 
\begin{pmatrix}    
0 & \cdots & 0 & 0 & \!\!\!\! -1\\ 
0 &  \cdots & 0  &\! -1 & x^1\\ 
\vdots & \iddots &\iddots & \iddots  &x^{{2}} \\ 
0&\!\!-1&x^{{1}}& \iddots &   \vdots    \\
-1&x^{{1}}&x^{{2}}  &\cdots & \!\!\! x^{{n-1}}
  \end{pmatrix}\  , \ \ L = \begin{pmatrix}
x^1 & 1 &  & & \\
x^2 & 0 & 1 & & \\
\vdots & \vdots & \ddots & \ddots & \\
x^{n-1} &  0 & \dots & 0 & 1\\
x^n & 0  & \dots & 0 & 0 
\end{pmatrix}.  
\end{equation} 
 These two examples are  related as follows: if one writes the pair \eqref{ex:LC} in the coordinates 
$(\sigma_1,\dots,\sigma_n)$, where $\sigma$'s  are  the coefficients of the characteristic polynomial $\det(t \,\Id-L)= t^n-\sigma_1t^{n-1}-\sigma_2t^{n-2}-\dots-\sigma_n$, one obtains (up to a sign) the pair \eqref{eq:gL}. Of course, the inverse transformation, from   (\ref{eq:gL})  to \eqref{ex:LC},  is possible only near those points where  $L$ has 
$n$  different  real-valued eigenvalues.

In both cases,  the metrics $g, gL^{-1},\dots, gL^{-n}$ are flat and $gL^{-n-1}$ has constant non-zero curvature. All together they generate a pencil of compatible Poisson brackets of hydrodynamic type of dimension $n+2$. It is not hard to see that this pencil is {\it maximal}. Our next result shows that a pencil with such properties is unique.

\begin{Th} \label{thm:main4}  Let $L$ be a Nijenhuis operator  which is invertible, self-adjoint with respect to a metric $g$ and differentially non-degenerate almost everywhere. Suppose the metrics 
$gL^{-k}$ are flat for $k=0,\dots,n$ and the metric $gL^{-n-1}$ has constant curvature $K\ne 0$. 
 Then  $g$ and 
$L$ are geodesically compatible. Moreover,  in a neighborhood of every  point, where $L$ is differentially non-degenerate, 
   the pair $(g, L)$ is locally isomorphic (up to multiplication of  $g$ by  a constant)  to \eqref{eq:gL}.
\end{Th}

\begin{Rem}{\rm
For dimension two, Theorem \ref{thm:main4} follows from  \cite[Theorem 10]{Ferapontov-viniti}. In arbitrary dimension,  this statement in a slightly different form was announced  in  \cite{Ferapontov-Pavlov}, but to the best of our knowledge, the proof has never been published.
}\end{Rem}

\begin{Rem} \label{rem:past}{\rm 
The case when also the metric $gL^{-n-1}$  is flat is much easier (provided $L$ is differentially non-degenerate and therefore is diagonalisable almost everywhere).  We discuss it within the proof, in Remark \ref{rem:future}. In this case essentially  only the following example and 
 its natural modifications  with complex-valued eigenvalues  are  possible: the metric is  $g= \operatorname{diag}(g_1(x^1),\dots,g_n(x^n))$ and the operator is $L = \operatorname{diag}(\ell_1(x^1),\dots, \ell_n(x^n))$. In this case, generators of commuting flows coming from this multihamiltonian structure are all functions of the Nijenhuis operator $L$ and are therefore Nijenhuis operators also.  In  this case, the metric $g$ is not geodesically   compatible with $L$.  Systems   of hydrodynamic type \eqref{sys:hidro} such that  $A$ is a Nijunhuis operator are well-understood in the case when $L$ is diagonalisable; they  decouple in Hopf equations and can be solved  (almost) explicitly.  An inclusion of singular points to this situation was done in \cite{Nijenhuis3}.

Note that  if we do not assume that $L$ is differentially non-degenerate, then there are non-trivial examples when the above explained process does not stop so that all the metrics of the form $gL^{-k}$, $k\in \mathbb N$ are flat and we obtain an infinite-dimensional Poisson pencils of hydrodynamic type (see Example \ref{ex:new}).}  \end{Rem}

%%%%%%%%%%%%%%%%%%%%%%%%%%%%%%%%%%%%%%%%%%%%%%%%%%%
\section{ Proof of Theorem \ref{thm:main1}} \label{sec:tensor}
%%%%%%%%%%%%%%%%%%%%%%%%%%%%%%%%%%%%%%%%%%%%%%%%%%%

Let $L$ be an operator geodesically compatible with a metric $g$ of any signature. We do not assume a priori that the metric is flat. We consider the function  $\lambda:= \frac{1}{2}\tr L$,  and its differential  $\mathrm{d}\lambda$ whose components  will be denoted by $\lambda_i =  \frac{\partial \lambda}{\partial x^i}$ and we denote the   components of  $\operatorname{grad}_g\lambda$   by $\lambda^i :=  g^{is }\lambda_s$.

We start with Lemma describing  the relationship between  Christoffel symbols of $\nabla$ and $\bar\nabla$, the Levi-Civita connections  of the metrics $g$ and $g = gL^{-1}$.

\begin{Lemma} \label{Lem:0}
Let $g$ and $L$ be geodesically compatible and $L$ be non-degenerate. Then the Christoffel symbols  of $\bar{g}= g L^{-1}$ are given by 
\begin{equation} \label{eq:Gamma}
\bar \Gamma^i_{jk}= \Gamma^i_{jk} - \lambda^i \bar{g}_{jk},
\end{equation}
where $\Gamma^i_{jk}$ are the Christoffel symbols of $g$.
\end{Lemma} 
\begin{proof}
We use the equation of geodesic compatibility (see e.g. \cite[Theorem 2]{BoMa2003}):
\begin{equation}\label{eq:L} 
\nabla_k L_{ij}= \lambda_i g_{jk} + \lambda_j g_{ik}, \quad \mbox{with $L_{ij} = L_i^s g_{s j}$},
\end{equation}
which we rewrite raising two  indexes by $g$ as 
\begin{equation} \label{eq:geodcomp} 
\nabla_k L^{ij} = \lambda^i\delta^j_k + \lambda^j \delta^i_k, \quad \mbox{with $L^{ij} = L^i_s g^{s  j}$}
\end{equation} 

Notice that $L^{ij} = \bar g^{ij}$  (where $\bar g^{is } \bar g_{s  j} = \delta^i_j$ as usual). Hence \eqref{eq:geodcomp} gives
$$
\pd{\bar{g}^{ij}}{u^k} + \bar{g}^{is}\Gamma^j_{sk} + \bar{g}^{js}\Gamma^i_{sk} = \lambda^i \delta^{j}_k + \lambda^j \delta^{i}_k
$$
or, equivalently,
$$
\pd{\bar{g}^{ij}}{u^k} + \bar{g}^{i s } \big( \Gamma_{s k}^j - \lambda^j \bar{g}_{s k}\big) + \bar{g}^{js } \big( \Gamma_{s k}^i - \lambda^i \bar{g}_{s k}\big)=0.
$$
This means that $\bar{g}$ is parallel w.r.t. the symmetric connection whose Christoffel symbols $\bar \Gamma^i_{jk}$ are defined by \eqref{eq:Gamma}. Thus, this is the Levi-Civita connection for $\bar{g}$, as stated.
\end{proof}

\begin{Cor}[Essentially, \cite{Sinjukov1}]
 \label{cor:1} Let $g$ be geodesically compatible with  a non-degenerate  operator  $L$. 
Then  $\bar g = gL^{-1}$ is geodesically compatible with  $L$. 
\end{Cor} 

\begin{proof}
Using the identity $\nabla g = 0$ and Lemma \ref{Lem:0} we get
\begin{equation*}
    \begin{aligned}
 \bar{\nabla}_k &(L^s_i \bar g_{s j} )  =    \bar{\nabla}_k g_{ij}  = \pd{g_{ij}}{u^k} - g_{j s }\bar{\Gamma}^s_{ik} - g_{js}\bar{\Gamma}^s_{ik}  \\   = g_{is } &\big(\Gamma_{jk}^s - \bar{\Gamma}_{jk}^s \big) + g_{js} \big(\Gamma_{ik}^s - \bar{\Gamma}_{ik}^s \big) 
     = g_{is} \lambda^s \bar{g}_{jk} + g_{js} \lambda^s \bar{g}_{ik} = \lambda_i \bar{g}_{jk} + \lambda_j \bar{g}_{ik}.
    \end{aligned}
\end{equation*}
It remains to notice that $\bar{\nabla}_k (L^s_i \bar g_{s j} ) = \lambda_i \bar{g}_{jk} + \lambda_j \bar{g}_{ik}$ is exactly the condition \eqref{eq:L} for $\bar{g}$ and $L$.
\end{proof}

\begin{Lemma}[\cite{Sinjukov1} in dimensions $n\ge 3$]  \label{Lem:1b} Assume that $g$ and $L$ are geodesically compatible, $g$  is flat and $L$ is  non-degenerate.  
 Then $\bar g = gL^{-1}$  has a  constant, possibly zero,    curvature. 
\end{Lemma}
\begin{proof}
For a metric $g$, we consider the following system of linear partial differential   equations  on the function $h$:
\begin{equation}\label{eq:tashiro}
\nabla_j \nabla^i h - \tfrac{1}{n} (\nabla^s \nabla_s h)  \delta^i_j =0.
\end{equation}
This equation naturally appeared in different parts of differential geometry  and  is well understood; let us recall the following known property  of this  equation:

{\it If the space of solutions of \eqref{eq:tashiro} is at least $n+1$-dimensional, then the  metric has constant curvature} (for $n\ge 3$ , see   \cite[end of page 343]{vries}). 

We did not find the  case $n=2$ in the literature 
so let us give a  proof.
Consider the canonical complex structure $J= J^i_j$  (which is the $g$-volume form with one  index raised) and observe that 
 \eqref{eq:tashiro}  implies that the ``skew gradient'' vector $J^i_s \nabla^s h$ is  Killing. Indeed, $\nabla_j (J^i_\beta \nabla^\beta h) = 
 J^i_\beta \nabla_j  \nabla^\beta h =  c J^i_j$ is $g$-skew-symmetric (with $c=\frac{1}{n} \nabla_s \nabla^s h$) so that the Killing equation is fulfilled.    The existence of three linearly independent solutions of \eqref{eq:tashiro} implies  the existence of 
 two linearly independent Killing vector fields so  the curvature is constant, as claimed.

For each  function  $h$ in view of \eqref{eq:Gamma},  its second 
$\bar g$-covariant derivative 
reads:
\begin{equation}
\label{eq:additional}
 \bar \nabla_j \bar \nabla_i  h =  \nabla_j  \nabla_i  h  +  \lambda^s (\nabla_s h)  \bar g_{ji}.
\end{equation}

Since $g$ is flat, locally it has $n+1$ linearly independent functions 
$h$ such that 
\begin{equation}\label{eq:cas}
\nabla_j \nabla_i h=0.
\end{equation}
In flat coordinates $y^1,\dots,y^n$  for $g$  (i.e., such that $\Gamma^i_{jk} =0$), these functions take the form $h=\alpha_0 + \alpha_1 y^1 + \dots + \alpha_n y^n$. 
In view of  \eqref{eq:additional} 
they also satisfy  $\bar\nabla_j \bar \nabla_i  h = \lambda^s (\nabla_s h) \, \bar g_{ji}$  which immediately implies \eqref{eq:tashiro} with respect to  the metric $\bar g$. 
Thus,   the dimension of the space of solutions of \eqref{eq:tashiro} for $\bar g$  is at least $n+1$  and hence
$\bar g$ has  constant curvature. 
\end{proof} 

Lemma \ref{Lem:1b}   proves  the first part of the first statement of Theorem \ref{thm:main1}.  
By construction, the operator $L$ is Nijenhuis  (as an operator geodesically compatible with $g$), so the property (A) from the definition of Poisson-comptible metrics is automatically fulfilled.  It remains to check property (B), we will do it at the end of this section.

Now let us discuss the  second statement of Theorem \ref{thm:main1}.    Since $L_1$ and $L_2$ are geodesically compatible with $g$, the metrics $g_1=gL_1^{-1}$ and $g_2=gL_2^{-1}$ have constant curvature by Lemma \ref{Lem:1b}.   Recall that operators compatible with a fixed metric $g$ are all Nijenhuis and form a vector space (Nijenhuis pencil). In particular,  $\alpha L_1 + \beta L_2$ is a Nijenhuis operator for all $\alpha,\beta\in\R$. We, however, need to check this property for the operator 
$(g_1)^{-1} g_2 = L_1 L_2^{-1}$.  To that end, we prove  the following general fact from Nijenhuis geometry which is important on its own.

\begin{Lemma}\label{almost}
Consider a pair of non-degenerate Nijenhuis operators $L_1$ and $L_2$. Then $\alpha L_1 + \beta L_2$ is Nijenhuis operator for arbitrary constants $\alpha, \beta\in\R$ if and only if $L_1L_2^{-1}$ is Nijenhuis operator.
\end{Lemma}
\begin{proof}
It is well known that if $L_1$ and $L_2$ are both Nijenhuis, then $\alpha L_1 + \beta L_2$ is Nijenhuis if and only if the Frolicher-Nijenhuis bracket  of $L_1$ and $L_2$ vanishes (see \cite{fn1} and e.g. \cite{Nijenhuis1}):
\begin{equation*}
    \begin{aligned}{}
    [[L_1 , L_2 ]] & = L_1 [L_2 v, w] + L_1 [v, L_2 w] + L_2 [L_1 v, w] + L_2 [v, L_1 w] - \\
    & - L_1 L_2 [v, w] - L_2 L_1 [v, w] - [L_1 v, L_2 w] - [L_2 v, L_1 w]    = 0. 
    \end{aligned}
\end{equation*}
Here $v, w$ are arbitrary vector fields and $[\, , \,]$ stands for standard Lie bracket of vector fields. Let us prove the following identity, which immediately implies the statement of the lemma:
$$
\mathcal N_{L_1} (v, w) + L_1L_2 ^{-1} L_1 L_2 ^{-1} \mathcal N_{L_2} (v, w) - \mathcal N_{L_1 L_2^{-1}} (L_2 v, L_2 w) = L_1 L_2 ^{-1} [[L_1 , L_2 ]](v, w).
$$
We have
\begin{equation*}
\begin{aligned}
    & \mathcal N_{L_1} (v, w) + L_1 L_2 ^{-1} L_1 L_2 ^{-1} \mathcal N_{L_2} (v, w) - \mathcal N_{L_1 L_2 ^{-1}} (L_2 v, L_2 w) = \\
    = & L_1 L_2 ^{-1} L_2 [L_1 v, w] + L_1 L_2 ^{-1} L_2 [v, L_1 w] - L_1 L_2 ^{-1} L_2 L_1 [v, w] - [L_1 v, L_1 w] + \\
    + & L_1 L_2 ^{-1} L_1 [L_2 v, w] + L_1 L_2 ^{-1} L_1 [v, L_2 w] - L_1 L_2 ^{-1} L_1 L_2 [v, w] - L_1 L_2 ^{-1} L_1 L_2 ^{-1}[L_2 v, L_2 w] - \\
    - & L_1 L_2 ^{-1}[L_1 v, L_2 w] - L_1 L_2 ^{-1}[L_2 v, L_1 w] + L_1 L_2 ^{-1} L_1 L_2 ^{-1} [L_2 v, L_2 w] + [L_1 v, L_1 w] = \\
    = & L_1 L_2 ^{-1} [[L_1 , L_2 ]](v, w),
\end{aligned}
\end{equation*}
as stated.
\end{proof}

To complete the proof of Theorem \ref{thm:main1}  (both items 1 and 2),  it remains to show that the curvature of the metric $gL^{-1}$ linearly depends on  $L$,  where $L$ is now understood as an element for the vector space (Nijenhuis pencil) of all operators geodesically compatible with $g$.  This  will be implied by the following Lemma:   

\begin{Lemma}\label{bolsadd1}
Let $g$ be a flat metric, $L$ be a geodesically compatible non-degenerate operator and 
 $K\in \R$ be the curvature of $gL^{-1}$.  

 Then, 
\begin{equation}\label{eq:hess_lambda}
\nabla^i \nabla_j \lambda  +  K \delta^i_j = 0.
\end{equation}
\end{Lemma}
\begin{proof}
This identity will be  derived from the following algebraic relation between the curvature tensor $R$ of a metric $g$ and an operator $L$ geodesically compatible to it, see e.g. \cite[Eq.  (13) and Theorem 7]{BolsKioMatv}: 
\begin{equation}\label{magic}
[ R(v,u), L] = [u\otimes g(v) - v\otimes g(u) ,  M], \quad \mbox{where }M^i_j = \nabla^{i} \nabla_{j} \lambda.
\end{equation}
Here $u$ and $v$ are arbitrary tangent vectors and $[~,~]$ denotes the standard matrix commutator $[A,B]=AB-BA$.  

W.l.o.g. we will assume that $L$ is not proportional to the identity (recall that if $L=f(x)\Id$, then $f(x)=\mathrm{const}$  \cite{Weil} and the statement becomes trivial). Since in our case $g$ is flat, then $R(v,u)=0$ implying $M=\rho\,\Id$ (as $M$ commutes with any matrix of the form $u\otimes g(v) - v\otimes g(u)$).  To show that $\rho = -K$, we apply \eqref{magic} once again for $\bar g=gL^{-1}$ and $L$ (as they are still geodesically compatible by Corollary \ref{cor:1}):
\begin{equation}\label{magic2}
[\bar R(v,u), L] = [u\otimes \bar g(v) - v\otimes \bar g(u),  
\bar M],\quad  \bar M^i_j = \bar\nabla^i\bar\nabla_j \lambda. 
\end{equation}
Since  $\bar g$ is of constant curvature $K$, then $\bar R(v,u) = K (v\otimes \bar g(u) - u\otimes \bar g(v))$.  On the other hand, formula  \eqref{eq:additional} gives:
$$
\bar \nabla^i \bar \nabla_j \lambda = L^i_s  \nabla^s  \nabla_j \lambda + \lambda_s\lambda^s\delta^i_j\quad\mbox{or, in our case,}\quad \bar M = \rho L +\lambda^s\lambda_s \Id.
$$
Hence, \eqref{magic2}  can be rewritten in the form
$$
[v\otimes \bar g(u) - u\otimes \bar g(v),  (K +  \rho) L ]=0 \quad \mbox{for all  $u$ and $v$}, 
$$
implying $K+\rho = 0$, that is $M^i_j= \nabla^{i} \nabla_{j} \lambda = -K \delta^i_j$, as required. 
\end{proof} 

Now consider two operators $L_1$ and $L_2$ compatible with $g$ and apply Lemma \ref{bolsadd1}  for the linear combination $L = \alpha L_1 + \beta  L_2$. 
We denote the curvatures of $L$, $L_1$ and $L_2$ by $K$, $K_1$, $K_2$, and similarly $\lambda = \frac{1}{2} \tr L$, $\lambda_1 = \frac{1}{2} \tr L_1$, $\lambda_2 = \frac{1}{2} \tr L_2$. Then Lemma \ref{bolsadd1} gives 
 $$
K \delta^i_j  = \nabla^i \nabla_j \lambda = \nabla^i \nabla_j \left(\alpha \lambda_1 +  \beta \lambda_2\right) = 
(\alpha K_1 + \beta K_2)\delta^i_j.
$$
Hence, $K = \alpha K_1 + \beta  K_2$, as required.  This completes the proof of the second statement of Theorem \ref{thm:main1}.  For the first statement, the end of proof is exactly the same. One just need to replace $L_1$ and $L_2$ with $\Id$ and $L$. Theorem \ref{thm:main1} is proved.

%%%%%%%%%%%%%%%%%%%%%%%%%%%%%%%%%%%%%%%%%%%%%%%%%%%%%%
\section{Proof of Theorem \ref{thm:main2}} \label{sec:4}
%%%%%%%%%%%%%%%%%%%%%%%%%%%%%%%%%%%%%%%%%%%%%%%%%%%%%%

We consider a flat metric $g$ and a non-degenerate operator $L$ geodesically compatible to it.  From Theorem \ref{thm:main1} we  already know 
that $\bar g = gL^{-1}$ is a metric of constant curvature  $K$. 

Recall that Casimir functions $h$ for $g$ are those satisfying the  equation
\begin{equation}
\label{Casforg}
\nabla^i\nabla_j h = 0
\end{equation}
whereas the equation for Casimirs of $\bar g=gL^{-1}$ is 
\begin{equation}
\label{Casforgbar}
\bar \nabla^i \bar \nabla_j h + K h \,\delta^i_j = 0,
\end{equation}

In view of \eqref{eq:additional},  the latter equation can also be rewritten in the form
\begin{equation}
\label{Casforgbar2}
\bar g^{s i} \nabla_s\nabla_j h + (\lambda^s \nabla_s h + Kh) \, \delta^i_j = 0
\end{equation}

We first show that the metrics $g$ and $\bar g$ have ``many'' common Casimirs.  The next Lemma is equivalent to the first statement of Theorem \ref{thm:main2}.

\begin{Lemma}\label{lem:bols2}
The vector space of Casimirs of $\bar g$ contains a subspace of codimension one that consists of Casimirs of $g$.  
\end{Lemma}

\begin{proof}  Let $h$ be a Casimir of $\bar g=gL^{-1}$.  Taking into account that $\bar g^{s i} = L^{s i}$,  we can rewrite \eqref{Casforgbar2} in the form
\begin{equation} 
\label{eq:Casimir1}
 L^{s i} \nabla_s\nabla_j h = \rho\,\delta^i_j \quad\mbox{or, equivalently,} \quad
 \nabla_s\nabla_j h = \rho (L^{-1})^r_j g_{rs} 
\end{equation}
with $\rho = -\lambda^s \nabla_s  h - Kh$.  Thinking of $\rho$ as an unknown function, we will derive some additional conditions for it.

Taking  $\nabla_k$ derivative  of the first formula in \eqref{eq:Casimir1}  and applying \eqref{eq:geodcomp} we obtain:
$$
(\lambda^s \delta^i_k + \lambda^i\delta^s_k) \nabla_s \nabla_j h  + L^{s i} \nabla_k\nabla_s\nabla_j h = \nabla_k \rho \, \delta^i_j.
$$ 
Now subtract the same relation with indices $j$ and $k$ interchanged,  use the fact that  $\nabla_k\nabla_s\nabla_j h =\nabla_j\nabla_s\nabla_k h$ (since $g$ is flat) and substitute $\nabla_i\nabla_j h$ from the second formula in \eqref{eq:Casimir1}:
$$
\rho \lambda_\beta (L^{-1})^\beta_j \, \delta^{i}_k - \rho \lambda_\beta (L^{-1})^\beta_k \, \delta^{i}_j = \nabla_k \rho \ \delta^i_j - \nabla_j\rho \ \delta^i_k.
$$
This implies $\nabla_k \rho = - \rho \, \lambda_\beta (L^{-1})^\beta_k$ or, equivalently, in more invariant way: 
$$
\mathrm{d}\rho = - \rho \, (L^{-1})^* \mathrm{d} \lambda.
$$

  If $\rho\equiv 0$, then \eqref{eq:Casimir1} becomes $\nabla_i\nabla_j h=0$, i.e. $h$ itself is a Casimir of $g$. Otherwise, 
  we use the following general property of Nijenhuis operators \cite[Proposition 2.2]{Nijenhuis1}
  \begin{equation}
  \label{propNij}
  L^* \mathrm{d} (\det L) = 2\det L \, \mathrm{d}\lambda,   
  \end{equation}
  which gives
  $
\frac{\mathrm{d} \rho}{\rho} 
= -  (L^{-1})^* \mathrm{d} \lambda  = - \frac{1}{2} \frac{ \mathrm{d} ({\det L})}{\det L}
$ 
and, finally, 
\begin{equation} \label{eq:Casimir3}
 \rho =  c_h \tfrac{1}{\sqrt{\det L}},  
\end{equation}
where $c_h$ is a constant function depending only on the choice of the Casimir function $h$ of $\bar g$.  In other words,  we obtain a natural map $h \mapsto c_h$ from the vector space of the Casimirs of $\bar g$ to $\R$.  This map is obviously linear so that the common Casimirs $h$ of $\bar g$ and $g$ are defined by one single linear relation $c_h=0$, which is equialent to the statement of the Lemma.  \end{proof}

Our next goal is to verify the second statement of Theorem \ref{thm:main2}.	

\begin{Lemma}  \label{Lem:1}  
Assume that $(g,L) $ is a geodesically compatible pair, the metric $g$ is flat and $\det L>0$. 
Then $ \sqrt{\det L}$ is a  Casimir of the Poisson structure corresponding to $gL^{-1}$. 
\end{Lemma} 
	
\begin{proof} We first rewrite relation \eqref{Casforgbar2}  for Casimirs of $\bar g= g L^{-1}$ in a slightly different way. Taking into account that $\bar g^{s i} = L^{s i}$ and using \eqref{eq:geodcomp},  we  obtain the following expression for the l.h.s. of \eqref{Casforgbar2}:
$$
\begin{aligned}
&\bar g^{s i} \nabla_s \nabla_j h + (\lambda^s\nabla_s h + Kh) \delta^i_j = \\
&L^{s i} \nabla_j\nabla_s h + (\lambda^s\nabla_s h + Kh) \delta^i_j = \\
&\nabla_j (L^{s i} \nabla_s h)  - (\nabla_j L^{s i}) \nabla_s h  + (\lambda^s\nabla_s h + Kh) \delta^i_j = \\
&\nabla_j ( L^{s i} \nabla_s h)  - (\lambda^s \delta^i_j + \lambda^i \delta^s_j ) \nabla_s h  + (\lambda^s \nabla_s h + Kh) \delta^i_j = \\
&\nabla_j ( L^{s i} \nabla_s h)  -  \lambda^i   \nabla_j h  +  Kh \delta^i_j 
\end{aligned}
$$

Thus, we need to verify that $h=\sqrt{\det L}$ satisfies 
\begin{equation}
\label{Casforgbar20}
\nabla_j ( L^{s i} \nabla_s h)  -  \lambda^i   \nabla_j h  +  Kh \delta^i_j = 0
\end{equation}
Once again we use the general property \eqref{propNij} of Nijenhuis operators that gives
$L^s_r \nabla_s \sqrt{\det L} = \lambda_r \sqrt{\det L} $ or, equivalently,   $L^{s i} \nabla_s \sqrt{\det L} = \lambda^i  \sqrt{\det L} $. Hence, for $h=\sqrt{\det L}$ we have:
$$
\begin{aligned}
 \nabla_j \left( L^{s i} \nabla_s \sqrt{\det L}\right)  -  \lambda^i   \nabla_j \sqrt{\det L}  +  K\sqrt{\det L} \,\delta^i_j &= \\ 
 \nabla_j (\lambda^i  \sqrt{\det L})  -  \lambda^i   \nabla_j\sqrt{\det L}   +  K \sqrt{\det L} \,\delta^i_j  &= \\
\lambda^i  \nabla_j \sqrt{\det L}   + (\nabla_j \lambda^i)  \sqrt{\det L}  -  \lambda^i   \nabla_j\sqrt{\det L}   +  K \sqrt{\det L}\, \delta^i_j &= \\
 \bigl(\nabla_j \lambda^i      +  K \delta^i_j  \bigr) \sqrt{\det L} &=0,
\end{aligned}
$$
as required (at the very last step we used Lemma \ref{bolsadd1}).
\end{proof}

The third statement of Theorem \ref{thm:main2} now immediately follows from Lemmas \ref{lem:bols2} and \ref{Lem:1}.  Indeed for any Casimir of $\bar g$, and in particular for $h=\sqrt{\det L}$,  formulas \eqref{eq:Casimir1}  and \eqref{eq:Casimir3} give:
\begin{equation} \label{eq:Casimirfinal1}
 \nabla^i \nabla_j h =   \frac{c_h}{\sqrt{\det (L)}}  (L^{-1})^{i}_{j}.        
\end{equation}
It remains to replace $L$ with the linear combination $-\alpha L + \Id$.

%%%%%%%%%%%%%%%%%%%%%%%%%%%%%%%%%%%%%%%%%%%%%%%%%%%%%%%%%%%%%
\section{Proof of Theorem \ref{thm:main3} and Corollary \ref{cor:compatible} and more properties of the related pencils}
\label{sec:proof of theorem 3} 
%%%%%%%%%%%%%%%%%%%%%%%%%%%%%%%%%%%%%%%%%%%%%%%%%%%%%%%%%%%%%

\begin{proof}[Proof of Theorem \ref{thm:main3}] Let $x^1,\dots, x^n$ be flat coordinates for the metric $g$. In these coordinates, the Christoffel symbols vanish so that the  covariant derivative coincides with the usual one. Then 
formula \eqref{eq:hess_lambda} reads 
\begin{equation}\label{corr1}
\tfrac{\partial  {\lambda}^i}{\partial x^j} = - K \delta^i_j
\end{equation}
with a constant $K$. Hence, 
\begin{equation} \label{corr2} \lambda^i=b^i- K x^i  \end{equation}  for some constants $b^1,\dots,b^n$.

Next, using  the fact that $\nabla_k= \frac{\partial}{\partial x^k}$, we obtain from \eqref{eq:geodcomp}: 
\begin{equation} \label{corr3}  
     \frac{\partial }{\partial x^k} L^{ij} = \lambda^i \delta^j_k + \lambda^j \delta^i_k \stackrel{\eqref{corr2}}{=}(b^i- K x^i ) \delta^j_k + (b^j- K x^j ) \delta^j_k. 
\end{equation}
This system of equations  implies   $L^{ij} = a^{ij} + b^i x^j + b^j x^i - K x^i x^j$  (with constants $a^{ij}$ that are necessarily symmetric with respect to $i$ and $j$) as we claimed.  Clearly every such $L^{ij}$ satisfies \eqref{corr3} so it is geodesically compatible with  $g$. 

To finish the proof, recall that near the points where $L^{ij}$ is non-degenerate, 
 the metric $gL^{-1}$ has constant curvature $K$  by Lemma \ref{bolsadd1}, and that $L^{ij}= g^{is}L_s^j$ is the contravariant inverse of $\bar g=gL^{-1}$. Theorem  \ref{thm:main3} is proved. \end{proof}

\begin{proof}[Proof of Corollary \ref{cor:compatible}]
As above, $g$ denotes a flat metric and $L$ is an operator geodesically compatible with $g$.
 Let $P$ and $Q$ be polynomials such that $P(L)$ and $Q(L)$ are non-degenerate.  The operator connecting the metrics $gP(L)^{-1}$ and $gQ(L)^{-1}$ is a rational function of the Nijenhuis operator $L$ and, therefore, is Nijenhuis also by \cite[Proposition 3.1]{Nijenhuis1}, so the first condition from the definition of compatibility is fulfilled. We need to  prove  the  ``curvature-additivity'' condition \eqref{eq:additivitycurvature}. 

In order to do it, we observe that in local coordinates from Theorem \ref{thm:main3},  the curvature tensors  of $gP(L)^{-1}$ and $gQ(L)^{-1}$ are given by rational functions in 
\begin{equation} N:= n + \frac{(n+1)n}{2} +  \frac{(n+1)(n+2)}{2}
\end{equation} 
variables: the first $n$ variables are $x^1,\dots,x^n$, the next $\tfrac{(n+1)n}{2}$ are $g^{ij}$ and the last   $\tfrac{(n+1)(n+2)}{2}$
are the data $(a^{ij}, b^i, K)$.  The ``curvature-additivity'' condition \eqref{eq:additivitycurvature} is then equivalent to 
a system of algebraic relations on these  $N$ variables.  If it is fulfilled in an non-empty open subset of  $\mathbb{R}^N$, it is fulfilled everywhere. 
Let us explain why such a subset exists. 

Recall that by \cite{Mokhov1,Mokhov} if two metrics $g$ and $\bar g$ are related by a Nijenhuis operator $L= L^i_j = \bar g^{is } g_{sj}$  with $n = \dim \ M$ distinct real  eigenvalues, then the  curvature tensors of $g$ and $\bar g$ satisfy  \eqref{eq:additivitycurvature} automatically\footnote{Conversely, the ``curvature-additivity'' condition \eqref{eq:additivitycurvature} becomes essential in the case of operators with multiple eigenvalues, which do appear in our setting, so that the direct statement from \cite{Mokhov1,Mokhov} is not formally applicable here.}.  This statement also follows from our proof of Theorem \ref{thm:main4}, see Remark  \ref{rem:last} below.   Next, notice that \eqref{ex:LC} gives us an example of $L$ having  $n$ different real eigenvalues. More precisely, let $\check g_{ij}, \check a^{ij}, \check b^i$ and $\check K$ be the data from Theorem \ref{thm:main3} corresponding to the example  \eqref{ex:LC} and at some point $\check x \in \mathbb{R}^n$ the corresponding $L$ has $n$ different real eigenvalues.  At every point of a small neighborhood  of  $(\check x^i, \check g_{ij}, \check a^{ij}, \check b^i,  \check K)$ in 
$\mathbb{R}^N$,  the corresponding operator $L$ still satisfies the property  that its eigenvalues are real and different. Using the above mentioned result from \cite{Mokhov1,Mokhov}, we conclude that in this neighbourhood,  \eqref{eq:additivitycurvature} is fulfilled which implies that it is fulfilled identically for all $L$ and at all points whenever it makes sense. 
\end{proof}

We conclude this section with discussing some more properties of the pencil $\mathcal L_g$ of Nijenhuis operators $L^i_j$ geodesically equivalent to  $g$(equivalently, pencil $\mathcal L$ of Poisson compatible contravariant metrics $L^{ij}$) from Theorem \ref{thm:main3}.  Notice that the relation between these two pencils can be written as $\mathcal L_g = \mathcal L g^{-1}$.  Also notice that $\mathcal L$ does not depend on the choice of $g$ whereas $L_g$ does.

First observe that $\dim\mathcal L_g$ equals $\frac{n(n+1)}{2} + n + 1 =   \frac{(n+2)(n+1)}{2}$, which is the dimension of the space of  symmetric $(n+1)\times (n+1)$-matrices. As the next theorem shows, this is not a coincidence: to each $L^{ij}$ of form \eqref{norm}  one can uniquely assign such a matrix by a  natural geometric procedure.

\begin{proposition} \label{lift} 
 Consider 
 $\mathbb{R}_{>0}\times \mathbb{R}^n$ (the coordinate on $\mathbb{R}_{>0}$  will be  denoted  by $x^0$ and those on $\mathbb{R}^n$ will be  $x^1,\dots,x^n$) with the symmetric affine  connection  $\hat \nabla =\left(  \hat  \Gamma^\alpha_{\beta \gamma}\right)$ (with $\alpha, \beta, \gamma \in \{0,\dots,n\}$)
 such that  the only non-zero Christoffel symbols are as follows:
\begin{equation} \label{gamma}\hat \Gamma^i_{j0}= \hat\Gamma^i_{0j}= \delta^i_j \frac{1}{x^0}  \textrm{ \  for $i,j=1,\dots,n$.} \end{equation}

Then $\hat\nabla$ is flat. Moreover, for any constants  $(a^{ij}, b^i, K)$ with 
$i,j =1,\dots,n$ and $a^{ij}=a^{ji}$, the symmetric contravariant $(2,0)$ tensor $A^{\alpha\beta}$ on $\mathbb{R}_{>0}\times \mathbb{R}^n$ 
  given by
\begin{equation}\label{lift1} 
 A^{ij} = \frac{1}{(x^0)^2} L^{ij},  \ \    A^{0i}=  -\frac{b^{i}- K x^i}{x^0}, \ \  A_{00}= -K  \quad (i, j = 1,\dots, n), 
\end{equation} 
 is parallel with respect to $\hat\nabla$.
Furthermore, in the ``new'' coordinates 
$$
y^0= x^0, \ \  y^i= x^0 x^i \quad   (i=1,\dots,n),  
$$
the Christoffel symbols of this connection vanish.  In these coordinates, the  matrix of $A$ becomes constant with components 
\begin{equation}\label{lift2} 
A^{ij} =  a^{ij} ,  \ \   A^{0i}= A^{i0}=  -b^{i},    \ \ A_{00}= -K \quad (i, j = 1,\dots, n).  
\end{equation}
\end{proposition} 

\begin{proof} One verifies Proposition \ref{lift}  by direct calculations. First we substitute the Christoffel symbols \eqref{gamma} into the formula \eqref{standard} for the curvature  tensor and see that it vanishes.  Next we substitute \eqref{lift1} in the relation $\hat \nabla A=0$ and see that it is equivalent to \eqref{corr3}. Again by direct calculations we observe that $\hat \nabla$-Hessians of the functions $y^\alpha$ vanish and, therefore, in the $y^\alpha$-coordinates 
the components of  $A$ are constants. Finally, calculating the Jacobi matrix $\Bigl(\frac{\partial y}{\partial x}\Bigr)$ and using it in the transformation rule for (2,0)-tensors  finishes the proof.
\end{proof} 

Note that the construction of the connection $\hat \nabla$ on $\mathbb{R}_{>0}\times \mathbb{R}^n$ does not involve $g$ or $L$, so it is the same for all $L$ and all $g$ from Theorem \ref{thm:main3}. In fact it is motivated by different `conification' procedures in projective (see e.g. \cite{Eastwood, Gover1,Gover}) and other Cartan  parabolic geometries. 

Proposition \ref{lift}  reduces the classification of geodesically compatible pairs  $(g, L)$ with flat $g$   
 to the classification of pairs of symmetric $(n+1)\times (n+1)$-matrices which is,  of course, well known (see e.g. \cite{LancRodm}). 
 Note that the metric $g$ itself is included in the family \eqref{norm} from  Theorem \ref{thm:main3} with $L=\Id$, i.e., $a^{ij}= g^{ij}$,  $b^i=0$ and  $K=0$. We see that  the first column and first row of the corresponding matrix $A_g$ vanish,  so that in terms of \cite{LancRodm} 
we are only interested in those pairs $A_g, A \in \operatorname{Sym}((n+1)\times (n+1))$ for which $A_g$ has rank $n$.

%%%%%%%%%%%%%%%%%%%%%%%%%%%%%%%%%%%%%%%%%%%

Our next observation is {\it maximality} of the pencils $\mathcal L_g$ and $\mathcal L$ and their relations with other (simpler) pencils. Let $x^1, \dots, x^n$ be canonical coordinates for the Euclidean metric $g \simeq \sum \mathrm{d} (x^i)^2$.  Consider the family $\mathcal S$ of operators $A$ given in these coordinates by symmetric matrices with constant entries.  Obviously every $A\in\mathcal S$ is a Nijenhuis operator and  therefore $\mathcal S$ is a Nijenhuis pencil.

Can one extend $\mathcal S$ to get a larger Nijenhuis pencil?  To answer this question,  we need to describe Nijenhuis operators $L$ such that  for each $A\in \mathcal S$ the sum $L + A$ is still Nijenhuis.   Analytically, this condition means that $L$ commutes with all $A\in\mathcal S$ in sense of Frolicher-Nijenhuis bracket \cite{fn1}, that is,
$$
\begin{aligned}
\, [[L, A]] (\xi,\eta) = &LA [\xi,\eta] - L[A\xi, \eta] - L[\xi, A\eta] + [L\xi, A\eta] +  \\
 + &AL [\xi,\eta] - A[L\xi, \eta] - A[\xi, L\eta] + [A\xi, L\eta]=0
\end{aligned}
$$
for any vector fields $\xi$ and $\eta$. If we temporarily ignore the fact that $L$ itself is a Nijenhuis operator, then the problem becomes linear and can be solved by straightforward computation that we omit.

\begin{proposition} \label{pro1}
Let  $[[ L, A]] = 0$ for all $A\in\mathcal S$. Then for $n\ge 3$, in the coordinates $x^1,\dots, x^n$, the matrix of $L$ takes the following form:
\begin{equation}\label{main_konyaev}
L =  A + x \, b^\top \!\! + c \, x^\top \!\! + K x \, x^\top, \ \   \mbox{where } \  x \!=\!\! \begin{pmatrix} x^1 \\ \vdots \\ x^n \end{pmatrix}\! , \ b \!=\!\! \begin{pmatrix} b^1 \\ \vdots \\ b^n \end{pmatrix}\! ,\  c \!=\!\! \begin{pmatrix} c^1 \\ \vdots \\ c^n \end{pmatrix}\! ,
\end{equation}
$A$ is an arbitrary constant matrix and  $K\in \R$  (cf. formula \eqref{norm} from Theorem \ref{thm:main3}). 

In dimension 2,  $L$ is the sum of an opertor \eqref{main_konyaev} and arbitrary skew-symmetic operator of the form  $\begin{pmatrix}  0 & f(x^1, x^2) \\ - f(x^1, x^2) & 0 \end{pmatrix}$. 
\end{proposition}

Notice that in order for \eqref{main_konyaev} to be a Nijenhuis operator, the parameters $A$, $b$, $c$ and $K$ must satisfy additional (non-linear) relations. However, the following fact is straightforward.

\begin{Cor} \label{cor:bols_1}
There is a unique Nijenhuis pencil $\mathcal L_g$ satisfying the following conditions:
\begin{itemize}
\item[{\rm(i)}]  $\mathcal L_g$ is maximal, 
\item[{\rm(ii)}]  $\mathcal L_g$ contains $\mathcal S$,
\item[{\rm(iii)}]  all operators from $\mathcal L_g$ are $g$-symmetric.
\end{itemize}
The matrices of operators $L\in\mathcal L_g$ in coordinates $x^1,\dots, x^n$ are given by  \eqref{main_konyaev} with $b = c$ and symmetric $A$.  In other words, $\mathcal L_g$ is the Nijenhuis pencil from Theorem \ref{thm:main3} with $g_{ij}=\delta_{ij}$. 
\end{Cor}

If we omit condition (iii), then there exist other maximal Nijenhuis extensions of $\mathcal S$.  For instance, we may take all operators $L$ of the form 
$L(x)= A + x b^\top$ with arbitrary constant $A$ (not necessarily symmetric!). 

For a flat metric $g$ of arbitrary signature, Corollary \ref{cor:bols_1} can be generalised as follows.  Let $\mathcal S$ be a Nijenhuis tensor that consists of covariantly constant $g$-symmetric operators $A$, then there exists a unique Nijenhuis pencil $\mathcal L_g$ satisfying the above conditions (i), (ii) and (iii).  This  pencil  coincides with that from Theorem \ref{thm:main3}.  The maximality of $\mathcal L_g$ immediately implies

\begin{Cor} \label{cor:bols_2}
The pencil $\mathcal L$ of Poisson-compatible metrics from Theorem  \ref{thm:main3} is maximal.
\end{Cor}

This pencil admits the following alternative description. Let $g$ be a flat metric and $\nabla$ be its Levi-Civita connection. Consider all the metrics $\hat g$ covariantly constant w.r.t. $\nabla$.  Obviously, they are flat and pairwise Poisson-compatible.  

\begin{Cor} \label{cor:bols_3}
Assume that $\bar g$ is almost compatible with every covariantly constant $\hat g$ in the sense of  \cite{Mokhov, Mokhov1}, i.e.,  $L=\bar g^{-1} \hat g$ is Nijenhuis.  Then
$\bar g$ has constant curvature. Moreover, any two metrics $\bar g_1$ and $\bar g_2$ satisfying this condition are Poisson compatible. All together they form the pencil $\mathcal L$ from  Theorem  \ref{thm:main3}. 
\end{Cor}

%%%%%%%%%%%%%%%%%%%%%%%%%%%%%%%%%%%%%%%%%%%%%%
\section{ Proof of Theorem  \ref{thm:einstein} and Poisson compatibility of metrics $g, gL^{-1}, gL^{-2},\dots$}
\label{sec:proof of theorem einstein}
%%%%%%%%%%%%%%%%%%%%%%%%%%%%%%%%%%%%%%%%%%%%%%

Throughout this section,  $g$ denotes a flat  metric  and  $L$ is an operator geodesically compatible with $g$.  

From Lemma \ref{bolsadd1} (and Theorem \ref{thm:main3}) we see that the metric $gL^{-1}$ is flat if and only if  $K=0$. The next two statements answer the  natural question on necessary and sufficient conditions for the metric  $gL^{-2}$, and more generally, all the metrics $gL^{-k}$ with $k\le k_0$ to be flat.  

\begin{Lemma} \label{lem:two}  Suppose $g$ is flat and 
$L$   is geodesically compatible with  $g$ and non-degenerate.  Let $gL^{-1}$ be flat, then $\lambda^j$ is a parallel vector field, i.e., $\nabla_i \lambda^j=0$, implying that   $g^{ij}\lambda_i \lambda_j$ is a constant. Moreover, 
  the metric $gL^{-2}$ is  flat if and only if   $\lambda^j$ is null, that is,   
$g_{ij}\lambda^i \lambda^j=g^{ij}\lambda_i\lambda_j=0$.   
\end{Lemma}

\begin{proof}
Since  $gL^{-1}$ is flat,  $K$ in Lemma \ref{bolsadd1} is zero. Hence, \eqref{eq:hess_lambda} implies that $\lambda^i$ is $\nabla$-parallel. Of course, the same is seen from  Theorem \ref{thm:main3}, since in this case $\lambda^i= a^i$ in flat coordinates and is clearly parallel.  

In order to prove the second claim, we use \eqref{eq:hess_lambda} for the metric $\bar g=gL^{-1}$ which is still compatible with $L$:
$\bar \nabla^i \lambda_j  + \bar K \delta^i_j = 0$.  
In view of \eqref{eq:additional}, we have  $\bar \nabla^i \lambda_j  = L^i_s  \nabla^s \lambda_j  - \lambda^s\lambda_s \delta^i_j$. Since
$\nabla^s \lambda_j = \nabla_j \lambda^s = 0$, we see that  $\bar K=0$  if and only if $\lambda^s \lambda_s=g^{is}\lambda_i\lambda_s=0$. 
\end{proof}

\begin{Cor} \label{cor:4}   Suppose $g$ is flat,  
$L$   is geodesically compatible with  $g$ and non-degenerate.  

The necessary and sufficient conditions for flatness of each of the metrics 
$g, gL^{-1}, \dots, gL^{-k}$ ($k\ge 2$) are as follows:   
the $1$-form $\mathrm{d}\left(\frac{1}{2}\operatorname{tr} L\right)=  \lambda_i$ is parallel with respect to $\nabla=\nabla^g$  and   
$$ 
g^{ij}\lambda_i \lambda_j =  
(gL^{-1})^{ij}\lambda_i \lambda_j=\cdots  = (gL^{-k+2})^{ij}\lambda_i \lambda_j=0.
$$ 
Moreover, if each of the metrics  
$g, gL^{-1}, \dots , gL^{-k}$  ($k\ge 2$)  is flat, then $\lambda_i$ is parallel with respect each Levi-Civita connection related to $g, gL^{-1},\dots,gL^{-k+1}$.  
\end{Cor}

Here by $(gL^{-m})^{ij}$ we denote the dual tensor to $(gL^{-m})_{ij}$. In other words, in matrix notation, $(gL^{-m})^{ij}$  are the components of the inverse matrix $(gL^{-m})^{-1} = L^m g^{-1}$.

\begin{proof}  Corollary \ref{cor:4} follows from iterative application of Lemma  \ref{lem:two}. 
\end{proof}

This corollary allows us to construct examples of pairs $(g, L)$  
such that $gL^{-k}$ is flat for all $k>0$ and $\mathrm{d}\operatorname{tr}(L)\ne 0$.  The simplest example of this kind is as follows.  Let $b=b^i\ne 0$ be a constant vector which is null with respect to a non-degenerate symmetric matrix  $g_{ij}$. We  view $g$  as a flat metric and 
take $a^{ij}=(g^{-1})$ and $K=0$.  The $(2,0)$-tensor $L^{ij}$ constructed by   \eqref{norm} for 
these $a^{ij},b^i, K$ evidently has the property that $\lambda^i$ is parallel with respect to $\nabla^g$  and  that  $ (gL^{-k})^{ij}\lambda_i \lambda_j=0$ for every $k$.

Let us give a more interesting example based on the same linear-algebraic idea.

\begin{Ex} \label{ex:new}  {\rm 
Take 
$$
g_{ij} =\begin{pmatrix}  & &  &1 \\ 
                   &  &  \iddots& \\
									& \iddots& & \\
									1& & & \end{pmatrix}\ , \ \ 
									a^{ij}=\begin{pmatrix}  & &  &0 \\ 
                   &  &  \iddots& 1\\
									& \iddots&\iddots & \\
									0& 1& & \end{pmatrix}\ , \ \ b= e_m \ , \ \ K=0.    
$$
(Here $e_m$ denotes the $m$-th basis vector so that  all components of $b$ vanish except for the $m$-th one which is equal to $1$). Now if $m\ge n/2 +1,$ then the necessary and sufficient flatness conditions for $g, gL^{-1}, \dots ,gL^{-k}$ from Corollary \ref{cor:4} are fulfilled for any $k$. This implies, in particular, that for any real analytic function $f$ of one variable, the metric $g (f(L))^{-1}$ is also flat (at those points where $f(L)$ is non-degenerate) so that we obtain an infinite-dimensional pencil of contravariant flat metrics of the form $f(L)g^{-1}$. 
}\end{Ex} 

This example  may possibly be of interest in the study of infinite-dimensional systems of hydrodinamic type for the following reason. In this theory one customary assumes that the operators defining these systems are diagonalisable and this special case is well studied.  The case of other Segre characteristics  is generally considered to be much harder and, in fact, there are only very few such examples in the literature.  In the above example,  the corresponding operator $L^i_j=L^{is}g_{sj}$  has nontrivial Jordan blocks. 
  For instance, in  the  4-dimensional case with $m=3$ and $g$ and $L$ given by:  
$$
g_{ij}= \begin{pmatrix}  & &  &1 \\ 
                   &  &  1& \\
									& 1& & \\
									1& & & \end{pmatrix}\ , \ 									
									L^{ij}=  \begin {pmatrix} 0&0&x^{{1}}&0\\0&0&x^{
{2}}&1\\ x^{1}&x^{{2}}&2\,x^{{3}}+1&x^{{4}}
\\ {}0&1&x^{{4}}&0\end {pmatrix},  
$$
the  operator  $L^i_j$ has two Jordan $2\times 2$ blocks with nonconstant eigenvalues. After a suitable coordinate transformation, $g$ and $L$ become 
$$
g_{ij} = \begin{pmatrix} 2\, \left( -2\,x^{{2}}-1 \right)  \left( 
x^{{3}}-x^{{1}} \right) & \left( x^{{3}}-x^{{1}} \right) ^{2}&0&0
\\  \left( x^{{3}}-x^{{1}} \right) ^{2}&0&0&0
\\ \noalign{\medskip}0&0&2\, \left( -2\,x^{{4}}-1 \right)  \left( x^{{
1}}-x^{{3}} \right) & \left( x^{{1}}-x^{{3}} \right) ^{2}
\\ \noalign{\medskip}0&0& \left( x^{{1}}-x^{{3}} \right) ^{2}&0
\end {pmatrix},$$ 
$$L^i_j=  
  \begin {pmatrix} x^{{1}}&0&0&0\\ 2\,x^{{
2}}+1&x^{{1}}&0&0\\ 0&0&x^{{3}}&0
\\ 0&0&2\,x^{{4}}+1&x^{{3}}\end {pmatrix}.
$$

This phenomenon survives for all dimensions. The appearing Jordan blocks are $2\times 2$, the number of such blocks depends on $m$ and in the case of even $n$ and $m=n/2+1$ equals $n/2$. The corresponding eigenvalues are not constant. 

As far as we know, such examples (with many Jordan blocks with nonconstant eigenvalues) were not known before and may open a door to 
other Segre characteristics in the theory of integrable systems of hydrodynamic type. 

Now, if in Example \ref{ex:new}  we set $m<n/2 +1$, then it is a simple exercise in Linear Algebra to check that for a certain $k_0$,   the metric $gL^{-k_0-1}$ has constant non-zero curvature, whereas
 $g,gL^{-1},\dots,gL^{-k_0}$ are still flat. The maximal value of such  
$k_0$ is $n$. Let us give the corresponding example.

\begin{Ex} \label{ex:1}{\rm
On $\mathbb{R}^n$ with coordinates $x^1,\dots ,x^n$, consider the  metric $g$  
 and the tensor $L^{ij}$ given by \eqref{norm} with  the following  data $ A, b$ and $K$:   
$$
g=\begin{pmatrix}  & &  &1 \\ 
                   &  &  \iddots& \\
									& \iddots& & \\
									1& & & \end{pmatrix}\ , \ \ 
									A=\begin{pmatrix}  & &  &0 \\ 
                   &  &  \iddots& 1\\
									& \iddots&\iddots & \\
									0& 1& & \end{pmatrix}\ , \ \ b = \begin{pmatrix} 1 \\ 0\\   \vdots \\ 0 \end{pmatrix} \ , \ \ K=0.    
$$
Then the metrics $gL^{-1},\dots,gL^{-n}$ are flat and $  gL^{-n-1}$ has constant non-zero curvature. 

}\end{Ex} 

In this example  $L^i_j$ has $n$ different nonconstant eigenvalues at a generic point. 
From Theorem \ref{thm:main4} it follows that this pair $(g, L)$ is locally isomorphic (modulo multiplication of $L$ by a constant) to that from  \eqref{eq:gL} (or   \eqref{ex:LC}, at almost every point).

%%%%%%%%%%%%%%%%%%%%%%%%%%%%%%%%%%%%%%%%%%%%%%%%%%%%%%%%%%%%%%%%

\begin{proof}[Proof of Theorem \ref{thm:einstein}]  
We assume that $g$ is Einstein with a constant  non-zero scalar curvature, $L$ is geodesically  compatible to $g$,    $\operatorname{det} L>0$ and  $\mathrm{d}\operatorname{tr}(L)\ne 0$. By $\tilde g$  we denote the metric $(\operatorname{det} L)^{-1} gL^{-1}$ that is 
geodesically equivalent to $g$.

Note that  $\tilde g$ is also Einstein (with constant  scalar curvature). In dimension $n>2$ it follows from \cite[Lemma 3 and Corollary 5]{einstein}, and in dimension $n=2$ it is Beltrami Theorem (see e.g. \cite{beltrami}). 

Next, consider the function $\phi= -\log \sqrt{\det L}$ and $\phi_i= \nabla_i \phi= -\mathrm{d}\left(\log \sqrt{\det L}\right)$. It is known (e.g., \cite[\S 2.2]{einstein}) that the Christoffel symbols and Ricci tensors  of $\tilde g$ and $g$  are related by the following formulas:
\begin{eqnarray}
\tilde \Gamma_{jk}^i &= & \Gamma_{jk}^i + \delta^i_j \phi_k +\delta^i_k \phi_j. 
\label{schouten1}
\\  \tilde R_{ij} &=&  R_{ij} - (n - 1)(\nabla_j \phi_{i} - \phi_i\phi_j).\label{schouten2}  \end{eqnarray}

 	Replacing  the $g$-covariant derivative $\nabla $ in \eqref{schouten2} by $\tilde g$-covariant derivative $\tilde \nabla$, we obtain 
	\begin{eqnarray}
	\tilde R_{ij} &=&  R_{ij} - (n - 1)(\tilde \nabla_j \phi_{i} + \phi_i\phi_j).\label{schouten3}  
	\end{eqnarray}
	Next, making the substitution $\phi= \log \psi$ (in our case $\psi= \left(\sqrt{\det L}\right)^{-1}$), we see that 
	\begin{equation} \label{schouten4}
	\tilde \nabla_j \phi_{i} + \phi_i\phi_j= \tilde \nabla_j \frac{\tilde \nabla_i \psi}{ \psi }+ \frac{\tilde \nabla_i \psi}{ \psi }\frac{\tilde \nabla_j \psi}{ \psi }= \frac{1}{ \psi }\tilde \nabla_j\tilde \nabla_i \psi.
	\end{equation}

Note  that the metric  $\bar g = gL^{-1}$ we are interesting in 
 is conformally related  to    $\tilde g=(\operatorname{det} L)^{-1} gL^{-1}$ and 
 the conformal coefficient is $\psi^{-2}  = e^{-2 \phi}$.
It is well-known (see for example \cite[eq. (2.21)]{brinkmann}) that the  Ricci tensors $\tilde R_{ij}$ and $\bar R_{ij}$ of two conformally related metrics $\tilde g$ and $\bar g= \psi^{-2} \tilde g = e^{-2\phi} \tilde g$  are connected   by
\begin{equation} \label{eq1}  
\bar R_{ij} =\tilde R_{ij}+ (\Delta_{\tilde g} \phi- (n-2) \| \tilde\nabla \phi\|_{\tilde g}^2)\tilde g_{ij}+ \frac{n-2}{\psi}\tilde \nabla_i\tilde \nabla_j \psi.   
\end{equation} 
Here $\Delta_{\tilde g} \phi:=  \tilde g^{ij} \tilde \nabla_i \tilde \nabla_j \phi$ is just the $\tilde g$-Beltrami-Laplace operator applied to  $\phi$.

Starting from this point, our proof depends on dimension. We first consider the case $n>2$.
 In this case  the first two terms on the right  hand side of \eqref{eq1} are proportional to $\tilde g$. If  $\bar g$ is Einstein, 
the left hand side is also proportional to $\tilde g$, so    $\tilde \nabla_i\tilde \nabla_j \psi$  is  proportional to $\tilde g$.
But this leads to a contradiction with \eqref{schouten3} and \eqref{schouten4}. Indeed, then  $R_{ij}$ is proportional to 
$\tilde g_{ij}$. Hence, $g_{ij}$ is proportional to $\tilde g_{ij}$, so $L^i_j$ is proportional to $\delta^i_j$. 
  By the  Weyl Theorem \cite{Weil}, this implies that $\mathrm{d}\operatorname{tr} L=0$ which is forbidden by our assumptions.

Let us now consider the remaining dimension $n=2$. From Corollaries 1 and 2  of  \cite{einstein} we see that   
if a metric $g$ is Einstein of constant scalar curvature  $(n-1) K$ and $L$ is compatible with it, then in addition to  \eqref{eq:L}
the following equation holds  for a certain constant $C$:
\begin{eqnarray} \label{VnB}
\nabla^i \lambda_j &=& (-K \tr L + C) \delta^i_j -K  L^i_j .
\end{eqnarray}

Note that though in \cite{einstein}  one generally assumes that $n= \textrm{dim} M\ge 3$,  formula \eqref{VnB}  holds in dimension 2 as well; in fact the statement 
$\nabla^i \lambda_j = \mu \delta^i_j -  K L^i_j$ (which is Corollary 1 in  \cite{einstein})
 follows from and is equivalent to \eqref{schouten2}, and the proof of the relation $\nabla_i \mu = -2 K \lambda_i$, which is Corollary 2 in \cite{einstein}, is straightforward and works also  in dimension 2.  Note also that  \eqref{VnB} for metrics of constant curvature 
 follows directly from \cite[Theorem 5.1]{Eastwood} or from \cite[Theorem 1]{Sinjukov2}.

Next, in view of \eqref{Lem:0} we have 
$$
\bar g^{si}  \bar  \nabla_{s} \lambda_j =   L^i_s  \nabla^s \lambda_j + \lambda^s \lambda_s \delta^i_j \stackrel{\eqref{VnB}}= (C-K \textrm{tr}L)L^i_j     -K (L^2)^i_j. 
$$
Now we use the fact that in dimension 2,  we have $(L^2)^i_j - \tr L  \cdot L^i_j  + \det L  \cdot\delta^i_j=0$ for any operator $L$ (Cayley-Hamilton theorem). This allows us to replace $(L^2)^i_j$ in the above formula  by $\tr L   \cdot L^i_j  - \det L \cdot  \delta^i_j$ and we obtain 
$$
\bar g^{si}  \bar  \nabla_{s} \lambda_j  = (\lambda^s\lambda_s+ K \det L) \delta^i_j +  (C- 2 K \tr L)  L^i_j. 
$$
We see that the left hand side of this formula is as in the formula \eqref{VnB} written for the metric $\bar g$; this implies that 
$(-2 K \tr L + C)=\const$ which is forbidden by the assumptions $K\ne 0$ and $\mathrm{d} \operatorname{tr} L\ne 0$.  The obtained contradiction proves Theorem  \ref{thm:einstein}.  
\end{proof}

%%%%%%%%%%%%%%%%%%%%%%%%%%%%%%%%%%%%%%%%%%%%%%%%%%%%%%%%%%%%%%%%%%

Combining Theorem \ref{thm:einstein} with  Lemma \ref{lem:two}   we immediately see that for positive definite flat  metrics  
the number $k_0$ such that  $g, gL^{-1},\dots, gL^{-k_0}$ are flat is at most two and the number $k_0$ such that  $g,gL^{-1},\dots, gL^{-k_0}$ are of constant curvature is at most three.

\begin{Cor} \label{cor:3} 
Suppose $g$ is a Euclidean metric (i.e., flat and positive definite).  
Let $L$  be  geodesically compatible with  $g$, non-degenerate and $\mathrm{d}\operatorname{tr} L\ne 0$.  

Then there is the following alternative: either $gL^{-1}$ has constant non-zero curvature and $gL^{-2}$ is not of constant curvature, or $gL^{-1}$ is flat, $gL^{-2}$ has constant non-zero curvature and $gL^{-3}$ is not of constant curvature. In particular, $gL^{-2}$ cannot be flat. 

\end{Cor}

\begin{proof} 
If $\lambda_i$ is  parallel, then in the positively definite case it is not null (recall that the condition $\mathrm{d}\operatorname{tr} L\ne 0$ implies that  $\lambda_i\ne 0$). Applying   Lemma  \ref{lem:two} we obtain the required claim. 
\end{proof}

%%%%%%%%%%%%%%%%%%%%%%%%%%%%%%%%%%%%
\section{Proof of Theorem \ref{thm:main4}} 
%%%%%%%%%%%%%%%%%%%%%%%%%%%%%%%%%%%%

\begin{proof}
Suppose $L$ is differentially-non-degenerate, then at almost every point it  has $n$ different eigenvalues. We will work in a small neighbourhood of such a point and    first consider the case when the eigenvalues are real. In this case, 
there exists a coordinate system $x^1,\dots,x^n$ such that $L$ is given by 
\begin{equation}  \label{eq:ourL}
\operatorname{diag}\left(x^1, x^2,\dots,x^n\right).\end{equation} 
 Let  the metrics $gL^{-k}$ with $k \in \{0,\dots,n+1\}$
satisfy the assumptions of Theorem \ref{thm:main4}. Since  the (0,2)-tensor $gL^{-1}$ is symmetric, $g$ is diagonal:
\begin{equation} \label{eq:g} 
g:= \operatorname{diag}\left(g_1, g_2,\dots,g_n\right),
\end{equation}
where $g_i=g_i(x^1,\dots,x^n)$ are some functions.  
Let us show that in this coordinate system the metric is given  (up to a constant) by \eqref{ex:LC}, which implies that $L$ is geodesically compatible to $g$.

We need the following technical Lemma:

\begin{Lemma}  \label{Lem:4} 
Let $g$  be diagonal as in \eqref{eq:g}   and 
$$
L= \operatorname{diag}\left({h_1(x^1)}, {h_2(x^2)},
\dots,{h_n(x^n)}\right)$$
with the function $h_i$ depending on the coordinate $x^i$ only. 

Then the Christoffel symbols $\Gamma^{i}_{jk}$ of the metric $g$ are  as follows:

\begin{itemize}
    \item $\Gamma^k_{ij} = 0$ for pairwise different $i, j$ and $k$, 
    \item $\Gamma^k_{kj} = \frac{1}{2} \frac{1}{g_k} 
		\frac{\partial g_k}{\partial x^j} $ for arbitrary $k, j$,
    \item $\Gamma^k_{jj} = - \frac{1}{2} \frac{1}{g_k} \frac{\partial g_j}{\partial x^k}$ for arbitrary  $k \neq j$.
\end{itemize}

Consequently, the Christoffel symbols $\bar \Gamma^{i}_{jk}$ of the metric $\bar g=  gL^{-1}$ are  as follows:

\begin{itemize}
    \item $\bar \Gamma^k_{ij} = 0$ for pairwise different $i, j$ and $k$, 
     \item $\bar \Gamma^k_{kj} = \Gamma^k_{kj} $ for arbitrary $k\ne j$,
		 \item $ \bar  \Gamma^k_{jj} = \frac{h_k}{h_j}\Gamma_{jj}^k $ for arbitrary $k \neq j$,
		\item $\bar \Gamma^k_{kk} =  \Gamma^k_{kk}- 
		\frac{1}{2} \frac{1}{h_k} 
		h_k'      $ for arbitrary $k$,  
\end{itemize}
 (each $h_k$ is a function of one variable, so $h_k'$  in the latter formula and below, in e.g. \eqref{eq:rijk}, is just the usual derivative, $h_k'=\tfrac{\partial h_k}{\partial x^k} $ ).
\end{Lemma} 

We leave the proof as a simple  exercise for the reader.  One needs  to substitute the components of $g$ in 
the formula $\Gamma^i_{jk}= \frac{1}{2} g^{si}\left( \tfrac{\partial g_{sk}}{\partial x^j}+ \tfrac{\partial g_{sj}}{\partial x^k} - \tfrac{\partial g_{jk}}{\partial x^s}\right)$ for the Christoffel symbols of $ g$ and then look how multiplications of $g_i$ with $1/h_i$ affects the formula.

\begin{Lemma} \label{Lem:5} 
The components of the curvature tensors $R^i_{\ jk\ell}$ and $\bar R^i_{\ jk\ell}$ of the connections $\Gamma$ and $\bar \Gamma$ from Lemma \ref{Lem:4} are  as follows: (no summation over repeating indices) 
\begin{itemize}
    \item $R^s_{\ ijk} = 0$ for arbitrary pairwise different $i, j, k$ and $s$, 
    \item $R^i_{\ ijk} = 0$ for arbitrary $i, j, k$,
    \item $R^j_{\ ijk} =  - \pd{\Gamma^j_{ij}}{x^k} + \Gamma^j_{ji} \Gamma^{i}_{ik} + \Gamma^j_{jk} \Gamma^{k}_{ik} - \Gamma^j_{kj} \Gamma^{j}_{ij}$ for arbitrary $j \neq i$ and $i \neq k$,
    \item $R^j_{\ iji} = \pd{\Gamma^j_{ii}}{x^j} - \pd{\Gamma^j_{ij}}{x^i} + \sum \limits_{\alpha = 1}^n \Gamma^j_{j \alpha} \Gamma^{\alpha}_{ii} - \Gamma^j_{ii} \Gamma^{i}_{ij} - \Gamma^i_{ij} \Gamma^i_{ij}$ for arbitrary $i$ and $j$,
    \item $R^k_{\ iji} = - R^k_{\ iij} = \pd{\Gamma^k_{ii}}{x^j} + \Gamma^k_{jj}\Gamma^{j}_{ii} + \Gamma^k_{jk} \Gamma^{k}_{ii} - \Gamma^k_{ii} \Gamma^{i}_{ij}$ for arbitrary $k \neq i$ and $k \neq j$.
\end{itemize}

\begin{itemize}
\item $\bar R^s_{\ ijk} = 0$ for arbitrary pairwise different $i, j, k$ and $s$,
\item $\bar{R}^i_{\ ijk} = 0$ for arbitrary $i, j, k$,
    \item $\bar{R}^j_{\ ijk} = R^j_{\ ijk}$ for arbitrary $j \neq i$ and $i \neq k.$
    \item For arbitrary $i\ne  j$,
		\begin{equation}\label{eq:rijk} \bar{R}^j_{\ iji} = \frac{h_j}{h_i} \pd{\Gamma^j_{ii}}{x^j} - \pd{\Gamma^j_{ij}}{x^i} + \sum \limits_{\alpha = 1}^n \frac{h_\alpha}{h_{i}} \Gamma^j_{j \alpha} \Gamma^{\alpha}_{ii} - \frac{h_j}{h_i} \Gamma_{ii}^j \Gamma^{i}_{ij} - \Gamma^j_{ij} \Gamma^j_{ij} -\frac{1}{2}\frac{h_i'}{h_i} \Gamma^j_{ij}+ \frac{1}{2  h_i}h_j' \Gamma^j_{ii}\ ,\end{equation} 
		
    \item $\bar{R}^k_{\ iji} = \frac{h_k}{h_i} R^k_{\ iji}$ \ .
\end{itemize}
\end{Lemma} 

We again 
leave  the proof of the Lemma as an exercise (which was done many times before us, see e.g.  \cite[\S 4]{Mokhov1}) 
for the reader. 
One  needs  to substitute $\Gamma^i_{jk}$ given by Lemma \ref{Lem:4} into the standard formula for the curvature
\begin{equation}  \label{standard} 
   R^\ell_{\ ijk} = \tfrac{\partial  }{\partial x^j} \Gamma^\ell_{ik} - \tfrac{\partial  }{\partial x^k} \Gamma^\ell_{ij}+ \Gamma^\ell_{js} \Gamma^s_{ik} - \Gamma^\ell_{ks} \Gamma^s_{ij}.   
\end{equation}
and then carefully implement the  changes in the  resulting  
formula induced by  replacing $\Gamma$ by $\bar \Gamma$ via formulas in Lemma \ref{Lem:4}.

\begin{Rem} \label{rem:convention} In our conventions used throughout the paper, the metric has  constant  curvature $K$, if (after raising the second index by $g$) it is given by:
$$
R^{ ij}_{\ \ km}= K \left(\delta^i_k \delta^j_m -\delta^i_m \delta^j_k\right).
$$
\end{Rem}

Note also that because of algebraic symmetries of curvature tensors, the components listed in Lemma \ref{Lem:5} are essentially all components.

By our assumptions, $g$ is flat which implies that  the components 
$\bar R^k_{\ iji}$ and $\bar{R}^j_{\ ijk}$  are zero. Thus, the    components of $\bar R$ which are ``interesting for us''
(in the sense that only they   have chance to be non-zero)   are those considered in 
\eqref{eq:rijk}.  It is  convenient to raise 
the index $i$ by $\bar g$ (since $\bar g$ is diagonal, this operation is just the  multiplication  
by $\tfrac{h_i}{g_i}$). 
We obtain:

\begin{equation} \label{eq:c} \bar{R}^{ij}_{\ \ ij} = 
\frac{h_j}{g_i}  \pd{\Gamma^j_{ii}}{x^j} - \frac{h_i}{g_i} \pd{  \Gamma^j_{ij}}{x^i} + 
\sum \limits_{\alpha = 1}^n \frac{h_\alpha}{g_{i}} \Gamma^i_{j \alpha} \Gamma^{\alpha}_{ii} - \frac{h_j}{g_i} \Gamma_{ii}^j \Gamma^{i}_{ij} - \frac{h_i}{g_i}\Gamma^j_{ij} \Gamma^j_{ij} -\frac{h_i'}{2 g_i} \Gamma^j_{ij} + \frac{h_j'}{2g_i} \Gamma^j_{ii}.\end{equation}

\begin{Rem} \label{rem:last}  {\rm To be used in the proof of Corollary \ref{cor:compatible}, let us observe   that  the equation  above is linear in $h_i$.  This 
was expected  in view of results of  \cite{Mokhov1,Mokhov}, and in fact the calculations above are similar to those in the corresponding places in \cite{Mokhov1}. The same is true if part of eigenvalues is complex-conjugated, see the discussion at the end of the proof. 
}\end{Rem}

Next, let us return
to our $L$ given by \eqref{eq:ourL} and  employ the condition that  $gL^{-k}$ are flat for $k= 0,\dots,n$ and $gL^{-n-1}$ has constant curvature $K$. 
Because the equations \eqref{eq:c} are linear in $h$, 
and since for a polynomial $P$ the matrix
 $P(L)$ 
is diagonal  
$\operatorname{diag}(P(x^1),\dots,P(x^n))$,  we have that  
for a polynomial $P$ 
  (such that $P(L)$ is non-degenerate)
the ``interesting components'' of the curvature tensor 
 of  the metric $gP(L)^{-1}$
are given by  \eqref{eq:c} with all $h_i$ replaced by $P(x^i)$.

By our assumptions 
 the metric   $gP(L)^{-1}$ is flat for  $P(t)= t^k$ 
 with  $k\le n$ and has constant curvature $K$ 
for $P(t)= t^{n+1}$. Using the linearity we can combine these conditions as follows: 
for any polynomial $P= a_{n+1} t^{n+1}+\dots+ a_0$ 
(of degree $\le n+1$) we have
\begin{equation} \label{eq:c1} \begin{array}{r}
\frac{P(x^j)}{g_i}  \pd{\Gamma^j_{ii}}{x^j} - \frac{P(x^i)}{g_i} \pd{  \Gamma^j_{ij}}{x^i} + 
\sum \limits_{\alpha = 1}^n \frac{P(x^\alpha)}{g_{i}} \Gamma^j_{j \alpha} \Gamma^{\alpha}_{ii} - \frac{P(x^j)}{g_i} \Gamma_{ii}^j \Gamma^{i}_{ij} \\ - \frac{P(x^i)}{g_i}\Gamma^j_{ij} \Gamma^j_{ij} -\frac{P'(x^i)}{2 g_i} \Gamma^j_{ij} + \frac{P'(x^j)}{2g_i} \Gamma^j_{ii}= a_{n+1} K. \end{array} \end{equation}

Let us view \eqref{eq:c1} as an equation on the functions $\Gamma^i_{jk}$, the following trick allows us  to find those $\Gamma^i_{jk}$ from this equation 
which are not automatically  zero by Lemma \ref{Lem:4}.

At  the  point $p= (\hat x_1,\dots,\hat x_n)$  as the polynomial $P$ we take  
$$
P_0(t)=  (t-\hat x_1)(t-\hat x_2)\dots (t-\hat x_n) \ \ \textrm{and} \ \   P_1(t)= t(t-\hat x_1)(t-\hat x_2) \dots (t-\hat x_n). 
$$ 
For both polynomials we have $P(\hat x_i)= 0$ for each $i$; for the polynomial $P_0$ we have 
$$
P'_0(\hat x_j) =  \prod_{s \ne j}(\hat x_j - \hat x_s) \ \textrm{and} \ \ P'_0(\hat x_i) =  \prod_{s \ne i}(\hat x_i - \hat x_s) 
$$
and for the polynomial $P_1$ we have  
$$
P'_1(\hat x_j) =  \hat x_j\prod_{s \ne j}(\hat x_j - \hat x_s) \ \textrm{and} \ \ P'_1(\hat x_i) =  \hat x_i\prod_{s \ne i}(\hat x_i - \hat x_s) .
$$
Substituting $P_0$ into \eqref{eq:c1} and evaluating the result at $p= (\hat x_1,\dots ,\hat x_n)$
 we see that most terms vanish and obtain the equation 
\begin{equation} \label{eq:c2} - \Gamma^j_{ij}\prod_{s \ne i}(\hat x_i - \hat x_s) +  \Gamma^j_{ii} \prod_{s \ne j}(\hat x_j - \hat x_s)= 0.  
\end{equation} 
Similarly, doing the same with $P_1$, we obtain 
\begin{equation} \label{eq:c3} -\hat x_i  \Gamma^j_{ij}\prod_{s \ne i}(\hat x_i - \hat x_s) + \hat x_j \Gamma^j_{ii} \prod_{s \ne j}(\hat x_j - \hat x_s)=  2 g_i K.  
\end{equation} 

The equations (\ref{eq:c2}, \ref{eq:c3}) are two linear equations on two unknowns $\Gamma^j_{ij}$ and $\Gamma^j_{ii}$.  Solving them at the point $p=(\hat x_1,\dots, \hat x_n)$ and taking into account that this point is arbitrary, we obtain:
 \begin{equation}\label{eq:gamma}
\Gamma^j_{ij}= \frac{2 g_iK}{(x^j-  x^i)\prod_{s \ne i}( x^i -  x^s)} \ , \ \ \Gamma^j_{ii}=\frac{2 g_i K}{( x^j-  x^i)\prod_{s \ne j}(x^j -  x^s)}.\end{equation}

Plugging them in the formulas from Lemma \ref{Lem:4}  we obtain (at every point $(x^1,\dots ,x^n)$): 

 \begin{eqnarray} \frac{2 g_iK}{(x^j- x^i)\prod_{s \ne i}( x^i -  x^s)}  &=& \frac{1}{2} \frac{1}{g_j} 
		\frac{ \partial g_j}{\partial x^i} \label{eq:gamma1}  \\ 
    \frac{2 g_iK}{(x^j-  x^i) \prod_{s \ne j}( x^j -  x^s)} &=& - \frac{1}{2} \frac{1}{g_j} \frac{\partial g_i}{\partial x^j} \label{eq:gamma2}\end{eqnarray}
(Notice that swapping $i$ and $j$ in \eqref{eq:gamma1} we obtain the formula equivalent   to \eqref{eq:gamma2} which was of course expected).

\begin{Rem}\label{rem:future}{\rm
From \eqref{eq:gamma1} we see that for  $K=0$ 
 each $g_i$ depends on the variable $x^i$ only, as we claimed in Remark \ref{rem:past}.   In this case, the pair 
 $(g,L)$ is not geodesically compatible.   }
\end{Rem} 

In what follows we  assume that the dimension $n \ge 3$. 
 Recall that  for dimension two Theorem \ref{thm:main4}  was proved in \cite[Theorem 10]{Ferapontov-viniti} and  in  \cite[Theorem 6.2]{Mokhov1}. 
 See also Remark \ref{proof:dim2} below.   Take $k\not\in \{i,j\}$ and consider the polynomial $P= \prod_{s\ne k}(t-\hat x_s)$ of degree $n-1$.
 For this polynomial, formula \eqref{eq:c1}  evaluated at the point $p=(\hat x_1,\dots ,\hat x_n)$ takes the form: 
\begin{equation} \label{eq:100}
P(\hat x_k) \Gamma^j_{jk} \Gamma^k_{ii} - \tfrac{1}{2} P'(\hat x_i) \Gamma^j_{ij} +   \tfrac{1}{2} P'(\hat x_j) \Gamma^j_{ii}=0.
\end{equation}
For our $P$, we have 
$$P'(\hat x_i)= \prod_{s\not\in\{i,k\}} (\hat x_i- \hat x_s) \ \ \textrm{and} \  \  P'(\hat x_j)= \prod_{s\not\in\{j,k\}} (\hat x_j- \hat x_s).$$

Substituting these and   \eqref{eq:gamma} into \eqref{eq:100}, 
 we  obtain: 
$$
\frac{4 g_i g_k K^2 }{(\hat x_j- \hat x_k)(\hat x_k- \hat x_i) \prod_{s\ne k}(\hat x_k- \hat x_s) }   -\tfrac{1}{2}  \frac{2 g_iK}{(\hat x_i - \hat x_k)(\hat x_j- \hat x_i)} +  \tfrac{1}{2} \frac{2 g_i K}{(\hat x_j - \hat x_k)(\hat x_j- \hat x_i)}=0.
$$
In this equation, the factor $g_i K$ cancels out and we obtain 
 the following  explicit  formula for $g_k K$: 
\begin{equation}\label{eq:Kg}
K g_k(\hat x_1,\dots,\hat x_n) = -\tfrac{1}{4}\prod_{s \ne k} (\hat x_k - \hat x_s).
\end{equation}
Since this formula is fulfilled at  every $p=(\hat x_1,\dots,\hat x_n)$, it remains true  if we replace all $\hat x_i$ by $x^i$. 
 We see that the metric is (up to a constant factor)  
as in example in Section \ref{sect2} so $g $  and $L$ are geodesically compatible as we claimed.

Let us now comment on the case when not all eigenvalues of $L$ are real. We may  again assume that all eigenvalues have algebraic multiplicity $1$.   Let $k$ of them are real and  the remaining $n-k$ are partitioned into complex conjugate pairs.

By \cite[Theorems 3.2, 3.3]{Nijenhuis1}, there exists 
a coordinate system   $(x^1,...,x^k, u^1, v^1,...,u^{\tfrac{n-k}{2}},  v^{\tfrac{n-k}{2}})$ such that 
 $L$ being written 
in the formal coordinates $$(x^1,...,x^k, z^1= u^1+ i v^1, \bar z^1= u^1-i v^1, ... , \bar z^{\tfrac{n-k}{2}}= u^{\tfrac{n-k}{2}}-i v^{\tfrac{n-k}{2}})$$ takes the form 
$$
L= \operatorname{diag}(x^1,...,x^k,z^1, \bar z^1,...,\bar z^{\tfrac{n-k}{2}}).
$$
Since $L$ is $g$-symmetric, in these coordinates we also have
$$
g= \operatorname{diag}(g_1,...,g_k, g_{k+1}, \bar g_{k+1},..., g_{\tfrac{n-k}{2}},  \bar g_{\tfrac{n-k}{2}}),
$$
where $g_i$  is a functions    of 
$(x^1,...,x^k,z^1,..., \bar z^{\tfrac{n-k}{2}})$ and $\bar g_i$ is  its complex conjugate.  Since $g$ is a usual real-valued metric, $g_1,...,g_k$ are real.

Let us denote the  formal  coordinates $z^1,..., \bar z^{\tfrac{n-k}{2}}$ by $x^{k+1},...,x^n$. 
Then, the Christoffel symbols and the curvature tensor are given by the same formulas as in the ``real'' case, so the analogs of Lemmas \ref{Lem:4}, \ref{Lem:5}
survive and we come to the system \eqref{eq:c1}. Solving the system in exactly same way as in the real case (where we used only differentiation of polynomials and algebraic manipulations, i.e., those operations that are perfectly defined over complex coordinates too) we conclude that the metric is given, up to a constant factor, by \eqref{ex:LC}. Next, observe that by \cite{BoMa2015} the metric and operator  of the form  
\eqref{ex:LC} are geodesically compatible even if $x^{k+1}= z^1,...,x^n= \bar z^{\tfrac{n-k}{2}}$, i.e., some pairs of coordinates are complex conjugate as shown. Thus, complex roots do not affect the construction and final conclusion.

Now assume that $L$ is given by \eqref{eq:gL}. Under the assumptions of Theorem \ref{thm:main4}, $L$ is geodesically compatible to $g$: indeed, above we have proved that it is geodesically compatible at almost every point which implies that it is geodesically compatible at every point. 
Since the metric $g$ is flat, the pair $(g,L)$ is real-analytic in some coordinates (this is known and  will follow from a description of geodesically compatible $(g,L)$ with $g$ flat).   As we explained  in Section \ref{sect2}, the form \eqref{ex:LC} in the coordinates $\sigma_1,...,\sigma_n$ is
 \eqref{eq:gL}. Theorem \ref{thm:main4} is proved. \end{proof}

\begin{Rem} \label{proof:dim2}{\rm
To be self-contained,  we sketch a proof of Theorem \ref{thm:main4}  in dimension 2. 
From   (\ref{eq:gamma1},\ref{eq:gamma2})
  we can find the derivatives $\tfrac{\partial g_1}{\partial x^2}$ and $\tfrac{\partial g_2}{\partial x^1}$ as functions of $g_1, g_2$:
	\begin{equation}\label{eq:sub}\begin{array}{rcl}	 {\frac {\partial }{\partial x^{{2}}}}{ g_1} &  = & {\frac {-4K{g_1}  {
g_2}  }{ \left( \hat x_{{1}}-\hat x_{{2}}
 \right) ^{2}}}\\{\frac {\partial }{\partial x^{{1}}}}{g_2} & = & {\frac {-4K{ g_1}  { g_2}  }{ \left( \hat x_{{1}}-\hat x_{
{2}} \right) ^{2}}}.  
\end{array}\end{equation}
		Further, we assume $K=1$ without loss of generality. 
		Plugging \eqref{eq:sub}  and also \eqref{eq:gamma} into  \eqref{eq:c1} we obtain  two linear equations in  the remaining derivatives $\tfrac{\partial g_1}{\partial x^1}$ and $\tfrac{\partial g_2}{\partial x^2}$. Solving them, we obtain:
		\begin{equation}\label{eq:sub1}\begin{array}{rcl}
		 {\frac {\partial {g_1}}{\partial x^{{1}}}} &= &2\,{\frac {{ g_1}  
 \left( \hat x_{{1}}-\hat x_{{2}}+2\,{ g_1}  
 \right) }{ \left( \hat x_{{1}}-\hat x_{{2}} \right) ^{2}}}\\{\frac {\partial{ g_2}  }{
\partial x^{{2}}}} &=&-2\,{\frac 
{{ g_2}   \left( \hat x_{{1}}-2\,{ g_2}
  -\hat x_{{2}} \right) }{ \left( \hat x_{{1}}-\hat x_{
{2}} \right) ^{2}}}.  
\end{array}\end{equation}
The equations (\ref{eq:sub},\ref{eq:sub1}) form a PDE system of Cauchy-Frobenius type (i.e., all derivatives of unknown functions 
 are explicitly given as algebraic expressions of unknown functions). Compatibility condition for this system is   
$$2\,{\frac { \left( 4\,{ g_2
}  -\hat x_{{1}}+\hat x_{{2}} \right)  \left( 4\,{
 g_1}  +\hat x_{{1}}-\hat x_{{2}} \right) }{
 \left( \hat x_{{1}}-\hat x_{{2}} \right) ^{4}}}=0
$$
which implies that either $g_2
  =\tfrac{1}{4}(\hat x_{{1}}-\hat x_{{2}})  $ or $g_1= \tfrac{1}{4}(\hat x_{{2}}-\hat x_{{1}})$.  Either of these formulas combined with \eqref{eq:sub} implies \eqref{eq:Kg}. 
}\end{Rem}

\end{document}